\documentclass[12pt]{amsart}
\usepackage{latexsym}
\usepackage{amssymb}
\usepackage{amscd}
\renewcommand\a{\alpha}
\renewcommand\b{\beta}
\newcommand\g{\gamma}

\newcommand\la{\lambda}

\newcommand\e{\eta}

\newcommand\io{\iota}

\newcommand\s{\sigma}

\newcommand\vf{\varphi}

\newcommand\ka{\kappa}

\newcommand\vS{\varSigma}
\newcommand\D{\Delta}
\newcommand\vD{\varDelta}

\newcommand\F{\Phi}

\newcommand\vL{\varLambda}

\newcommand\ve{\varepsilon}

\newcommand{\QQ}{\mathbb Q}

\newcommand{\ZZ}{\mathbb Z}

\newcommand{\CC}{\mathbb C}

\newcommand\BP{\mathbf P}

\newcommand\BF{\mathbf F}

\newcommand\Ba{\mathbf a}

\newcommand\Bp{\mathbf p}
\newcommand\Bm{\mathbf m}

\newcommand\Bs{\mathbf s}
\newcommand\Bk{\mathbf k}
\newcommand\Bh{\mathbf h}

\newcommand\ZC{\mathcal{C}}
\newcommand\CH{\mathcal{H}}

\newcommand\CL{\mathcal{L}}

\newcommand\CS{\mathcal{S}}

\newcommand\CG{\mathcal{G}}
\newcommand\CT{ \mathcal{T}}

\newcommand\FS{\mathfrak S}

\newcommand\Fg{\mathfrak g}

\newcommand\Fs{\mathfrak s}
\newcommand\Fl{\mathfrak l}

\newcommand\Fp{\mathfrak p}

\newcommand\Fsl{\Fs\Fl}
\newcommand\iv{^{-1}}
\newcommand\wh{\widehat}
\newcommand\wt{\widetilde}

\newcommand\we{\wedge}

\newcommand\ol{\overline}

\newcommand\lra{\leftrightarrow}

\newcommand\hinf{\frac{\infty}{2}}

\newcommand\Hom{\operatorname{Hom}}
\newcommand\End{\operatorname{End}}

\newcommand\lp{\operatorname{\!\langle\!}}
\newcommand\rp{\operatorname{\!\rangle\!}}

\newcommand{\isom}{\,\raise2pt\hbox{$\underrightarrow{\sim}$}\,}
\numberwithin{equation}{section}

\newtheorem{thm}{Theorem}[section]
\newtheorem{lem}[thm]{Lemma}
\newtheorem{cor}[thm]{Corollary}
\newtheorem{prop}[thm]{Proposition}

\def \para#1{\par\medskip\textbf{#1}
              \addtocounter{thm}{1}}

\def \remark#1{\par\medskip\noindent
                \textbf{Remark #1}
                \addtocounter{thm}{1}}

\begin{document}
\setlength{\baselineskip}{4.9mm}
\setlength{\abovedisplayskip}{4.5mm}
\setlength{\belowdisplayskip}{4.5mm}
\renewcommand{\theenumi}{\roman{enumi}}
\renewcommand{\labelenumi}{(\theenumi)}
\renewcommand{\thefootnote}{\fnsymbol{footnote}}
\renewcommand{\thefootnote}{\fnsymbol{footnote}}
\allowdisplaybreaks[2]
\parindent=20pt
\medskip
\begin{center}
{\bf Product formulas for the cyclotomic $v$-Schur algebra  \\
      and for the canonical bases of the Fock space} 
\\
\vspace{1cm}
Toshiaki Shoji and Kentaro Wada\footnote{
Both authors would like to thank B. Leclerc for 
valuable discussions.} 
\\
\vspace{0.7cm} 
{\it To Gus Lehrer on the occasion of his 60th birthday}
\\
\vspace{0.7cm}
Graduate School of Mathematics \\
Nagoya University  \\
Chikusa-ku, Nagoya 464-8602,  Japan
\end{center}
\title{}
\maketitle
\begin{abstract}
Let $\BF_q[\Bs]$ be the $q$-deformed Fock space of level $l$
with multi-charge $\Bs = (s_1, \dots, s_l)$. Uglov defined 
canonical bases $\CG^{\pm}(\la, \Bs)$ of $\BF_q[\Bs]$ for 
$l$-partitions $\la$, 
and the polynomials $\vD^{\pm}_{\la\mu}(q) \in \ZZ[q^{\pm}]$ are 
defined as the coefficients of the transition matrix between the 
canonical bases and the standard basis of $\BF_q[\Bs]$.
In this paper, we prove a product formula for 
$\vD^{\pm}_{\la\mu}(q)$ for certain $l$-partitions $\la, \mu$, 
which expresses $\vD_{\la\mu}^{\pm}(q)$ as a product of 
such polynomials related to various smaller Fock spaces
 $\BF_q[\Bs^{[i]}]$. 
Yvonne conjectures that $\vD_{\la\mu}^+(q)$ are related to the 
$q$-decomposition numbers $d_{\la\mu}(q)$ of the cyclotomic 
$v$-Schur algebra $\CS(\vL)$, where the 
parameters $v, Q_1, \dots, Q_l$ are roots of unity in $\CC$
determined from $\Bs$. In our earlier work, we have proved 
a product formula for $d_{\la\mu}(q)$ of $\CS(\vL)$, and the 
product formula for $\vD_{\la\mu}^{\pm}(q)$ is regarded 
as a counter-part of that formula for the case of the Fock space. 

\end{abstract}
\pagestyle{myheadings}
\markboth{SHOJI}{PRODUCT FORMULAS}

\bigskip
\medskip
\addtocounter{section}{-1}
\section{Introduction}
For a given integer $l>0$, let $\Pi^l$ be the set of 
$l$-partitions $\la = (\la^{(1)}, \dots, \la^{(l)})$.
Here $\la^{(i)} = (\la^{(i)}_1 \ge \la^{(i)}_2 \ge \cdots )$
is a partition for $i = 1,\dots, l$.  We denote by 
$|\la| = \sum_{i=1}^l|\la^{(i)}|$, and call it the size of $\la$, 
where $|\la^{(i)}| = \sum_{j} \la^{(i)}_j$.
For given integers $l, n \ge 1$ and an indeterminate $q$, we consider the 
$q$-deformed Fock space $\BF_q[\Bs]$ of level $l$
with multi-charge $\Bs = (s_1, \dots, s_l) \in \ZZ_{\ge 1}^l$, 
which is a vector
space over $\QQ(q)$ with the standard basis 
$\{ |\la, \Bs\rp \mid \la \in \Pi^l \}$, equipped with an action of
the affine quantum group $U_q(\wh {\Fs\Fl}_n)$.  
In [U], Uglov constructed the canonical
bases $\CG^{\pm}(\la, \Bs)$ for $\BF_q[\Bs]$ with respect to
$U_q(\wh{\Fs\Fl}_n)$, and for each 
$\la,\mu \in \Pi^l$, 
he defined a polynomial $\vD_{\la,\mu}^{\pm}(q) \in \QQ[q^{\pm 1}]$ 
by the formula
\begin{equation*}
\CG^{\pm}(\la,\Bs) = \sum_{\mu \in \Pi^l}
         \vD_{\la\mu}^{\pm}(q)|\mu,\Bs\rp,
\end{equation*}
where $\vD^{\pm}_{\la\mu} = 0$ unless $|\la| = |\mu|$.
\par
The cyclotomic $v$-Schur algebra $\CS(\vL)$
with parameters $v, Q_1, \dots, Q_l$ over a filed $R$, associated 
to $\CH_{N,l}$ was introduced by 
Dipper-James-Mathas [DJM], where $\CH_{N,l}$ is the Ariki-Koike
algebra associated to the complex reflection group 
$\FS_N\ltimes (\ZZ/l\ZZ)^N$.  Let $\vL^+$ be the set 
of $\la \in \Pi^l$ such that $|\la| = N$.
$\CS(\vL)$ is a cellular algebra in the sense of 
Graham-Lehrer [GL], and the Weyl module
$W^{\la}$ and its irreducible quotient $L^{\la}$ are defined 
for each $\la \in \vL^+$.  The main problem in the 
representation theory of $\CS(\vL)$ is the determination 
of the decomposition numbers $[W^{\la} : L^{\mu}]$ for 
$\la, \mu \in \vL^+$.
By making use of the Jantzen filtration of $W^{\la}$, the 
$q$-decomposition number $d_{\la\mu}(q)$ is defined, which is
a polynomial analogue of the decomposition number, and we have
$d_{\la\mu}(1) = [W^{\la} : L^{\mu}]$.
\par
Let $\Bp = (l_1, \dots, l_g)$ be a $g$-tuple of positive
integers such that $l_1 + \cdots + l_g = l$.
For each $\mu = (\mu^{(1)}, \dots, \mu^{(l)}) \in \vL^+$, 
one can associate a $g$-tuple of multi-partitions 
$(\mu^{[1]}, \dots, \mu^{[g]})$ by using $\Bp$, where
$\mu^{[1]} = (\mu^{(1)},\dots, \mu^{(l_1)}), 
\mu^{[2]} = (\mu^{(l_1+1)}, \dots, \mu^{(l_1+l_2)})$ and so on.
We define $\a_{\Bp}(\mu) = (N_1, \dots, N_g)$, where 
$N_i = |\mu^{[i]}|$.
Hence for $\la, \mu \in \vL^+$, $\a_{\Bp}(\la) = \a_{\Bp}(\mu)$
means that $|\la^{[i]}| = |\mu^{[i]}|$ for $i = 1, \dots, g$.
Let $\CS(\vL_{N_i})$ be the cyclotomic $v$-Schur algebra 
associated to $\CH_{N_i, l_i}$ with parameters
$v, Q_1^{[i]}, \dots, Q_{l_i}^{[i]}$, where 
$Q^{[1]}_1 = Q_1, \dots,Q^{[1]}_{l_1} = Q_{l_1}, 
Q^{[2]}_1 = Q_{l_1+1}, Q^{[2]}_2 = Q_{l_1+2}, \dots$.
In [SW], the product formula for the decomposition numbers of
$\CS(\vL)$ was proved, and it was extended in [W] 
to the product formula for $q$-decomposition numbers, which 
is given as follows; for $\la, \mu \in \vL^+$ such that 
$\a_{\Bp}(\la) = \a_{\Bp}(\mu)$, we have 
\begin{equation*}
\tag{*}
d_{\la\mu}(q) = \prod_{i=1}^g d_{\la^{[i]}\mu^{[i]}}(q),
\end{equation*}
where $d_{\la^{[i]}\mu^{[i]}}(q)$ is the $q$-decomposition number 
for $\CS(\vL_{N_i})$.
\par
We assume that the
parameters are given by 
$(v; Q_1, \dots, Q_l) = (\xi, \xi^{s_1}, \dots, \xi^{s_l})$, 
where $\xi$ is a primitive $n$-th root of unity in $\CC$ 
and $\Bs = (s_1, \dots, s_l)$ is a multi-charge. 
For an integer $M$, we say that $|\la, \Bs\rp$ is $M$-dominant
if $s_i - s_{i+1} > |\la| + M$ for $i = 1, \dots, l-1$.
Yvonne [Y] gave a conjecture that $d_{\la\mu}(q)$ coincides with
$\vD^+_{\mu^{\dag}\la^{\dag}}(q)$ if $|\la, \Bs\rp$ is 0-dominant,
where $\la^{\dag}, \mu^{\dag}$ are certain elements in $\vL^+$ 
induced from $\la, \mu$ (see Remark 2.7).
\par
In view of Yvonne's conjecture, it is natural to expect 
a formula for $\vD^+_{\la\mu}(q)$ as a counter-part for the Fock 
space of the product formula for $d_{\la\mu}(q)$.  
In fact, our result shows that it is certainly the case. 
We write 
$\Bs = (\Bs^{[1]}, \dots, \Bs^{[g]})$ with 
$\Bs^{[1]} = (s_1, \dots, s_{l_1}), \Bs^{[2]} = 
    (s_{l_1+1},\dots, s_{l_1+l_2})$, and so on.  Let 
$\BF_q[\Bs^{[i]}]$ be the $q$-deformed Fock space of 
level $l_i$ with multi-charge $\Bs^{[i]}$ 
for $i = 1, \dots, g$.  Then one can define polynomials 
$\vD^{\pm}_{\la^{[i]}\mu^{[i]}}(q)$ for $\BF_q[\Bs^{[i]}]$
similar to $\BF_q[\Bs]$.
The main result in this paper 
is the following product formula (Theorem 2.9); assume that 
$|\la, \Bs\rp$ is $M$-dominant for $M > 2n$.  Then for 
$\la, \mu \in \vL^+$ such that $\a_{\Bp}(\la) = \a_{\Bp}(\mu)$, 
we have
\begin{equation*}
\tag{**}
\vD^{\pm}_{\la\mu}(q) 
    = \prod_{i=1}^g\vD^{\pm}_{\la^{[i]}\mu^{[i]}}(q).
\end{equation*}
\par
Yvonne's conjecture implies, by substituting $q = 1$, that
$[W^{\la}: L^{\mu}] = \vD^+_{\mu^{\dag}\la^{\dag}}(1)$, which
we call LLT-type conjecture.  
In the case where $l = 1$, i.e., the case where $\CS(\vL)$ is
the $v$-Schur algebra of type $A$, it is known by 
Varagnolo-Vasserot [VV] that LLT-type conjecture holds. 
By applying the formulas (*), (**) to the case where 
$\Bp = (1, \dots, 1)$, we obtain the following partial result
for the conjecture; 
we consider the general $\CS(\vL)$, but 
assume that $|\la^{(i)}| = |\mu^{(i)}|$ for $i = 1,\dots, l$.
Then LLT-type conjecture holds (under a stronger dominance condition) 
for $[W^{\la} :  L^{\mu}]$.
\par
This paper is organized as follows; in Section 1, we give a 
brief survey on the product formula for $\CS(\vL)$ based 
on [SW], [W], which is a part of the first author's talk 
at the conference in Canberra, 2007. In Section 2 and 3, 
we prove the product formula (**) for the canonical bases
of the Fock space.    

\section{Product formula for the cyclotomic $v$-Schur algebra}
\para{1.1.}
Let $\CH = \CH_{N,l}$ be the Ariki-Koike algebra over an integral
domain $R$ associated to the complex reflection group 
$W_{N,l} = \FS_N\ltimes (\ZZ/l\ZZ)^N$ with parameters 
$v, Q_1, \dots, Q_l \in R$ such that $v$ is invertible, 
which is an associative algebra with 
generators $T_0, T_1, \dots, T_{N-1}$ and relations
\begin{align*}
&(T_0 -Q_1)\cdots (T_0 - Q_l) = 0 \\
&(T_i - v)(T_i + v\iv) = 0 \qquad\text{ for $i = 1, \dots, N-1$}
\end{align*}
with other braid relations.
The subalgebra generated by $T_1, \dots, T_{N-1}$ is isomorphic 
to the Iwahori-Hecke algebra associated to the symmetric group 
$\FS_N$, and has a basis $\{ T_w \mid w \in \FS_N\}$, where 
$T_w = T_{i_1}\dots T_{i_r}$ for a reduced expression 
$w = s_{i_1}\dots s_{i_r}$. of $w$.
\par
An element $\mu = (\mu_1, \dots, \mu_m) \in \ZZ^m_{\ge 0}$
is called a composition of $|\mu|$ consisting of $m$-parts, where 
$|\mu| = \sum_{i=1}^m\mu_i$.  A composition $\mu$ is called a partition 
if $\mu_1 \ge \cdots \ge \mu_m \ge 0$.
An $l$-tuple of compositions (resp. partitions) 
$\mu = (\mu^{(1)}, \dots, \mu^{(l)})$
is called an $l$-composition (resp. an $l$-partition) of 
$|\mu|$, where $|\mu| = \sum_i |\mu^{(i)}|$. 
\para{1.2.}
We define, following [DJM], 
a cyclotomic $v$-Schur algebra $\CS(\vL)$ associated 
to $\CH$. 
Fix $\Bm = (m_1, \dots, m_l) \in \ZZ^l_{>0}$ and let 
$\vL = \vL_{N,l}(\Bm)$ (resp. $\vL^+ = \vL^+_{N,l}(\Bm)$) be
the set of $l$-compositions (resp. $l$-partitions) 
$\mu = (\mu^{(1)}, \dots, \mu^{{(l)}})$
of $N$ such that $\mu^{(i)}$ has $m_i$-parts. 
For each $\mu \in \vL$, we define an element $m_{\mu} \in \CH$
as follows;
define $L_k \in \CH$ ($1 \le k \le N$) by $L_1 = T_0$ 
and by   
$L_k = T_{k-1}L_{k-1}T_{k-1}$ for
$k = 2, \dots, N$.  Then $L_1, \dots, L_N$ 
commute each other.  For each $\mu \in \vL$, we define 
$\Ba = \Ba(\mu) = (a_1, \dots, a_l)$ by 
$a_k = \sum_{i=1}^{k-1}|\mu^{(i)}|$ for $k = 2, \dots, l$, and 
by $a_1 = 0$.  Put 
\begin{equation*}
u_{\Ba}^+ = \prod_{k=1}^l\prod_{i = 1}^{a_k}(L_i - Q_k), 
\qquad x_{\mu} = \sum_{w \in \FS_{\mu}}v^{l(w)}T_w,
\end{equation*}
where $l(w)$ is the length of $w \in \FS_n$, and 
$\FS_{\mu}$ is the Young subgroup of $\FS_n$ corresponding to 
$\mu$.  
Then $u_{\Ba}^+$ commutes with $x_{\mu}$, and we put 
$m_{\mu} = u_{\Ba}^+x_{\mu}$.
\par
For each $\mu \in \vL$, 
we define a right $\CH$-module $M^{\mu}$ by $M^{\mu} = m_{\mu}\CH$, and 
put $M = \bigoplus_{\mu \in \vL}M^{\mu}$.  
Then the cyclotomic $q$-Schur algebra
$\CS(\vL)$ is defined as 
\begin{equation*}
\CS(\vL) = \End_{\CH}M = 
       \bigoplus_{\nu, \mu \in \vL}\Hom_{\CH}(M^{\mu}, M^{\nu}).
\end{equation*}
\para{1.3.} 
For each $\la \in \vL^+$ and $\mu \in \vL$, the notion of
semistandard tableau of shape $\la$ and type $\mu$ was introduced by 
[DJM], extending the case of a partition $\la$ and a composition $\mu$.
We denote by $\CT_0(\la, \mu)$ the set of semistandard tableau of 
shape $\la$ and type $\mu$ for $\la \in \vL^+, \mu \in \vL$.
The notion of dominance order for partitions is also generalized for
$\vL$, which we denote by $\mu \triangleleft \nu$.  Note that
$\CT_0(\la, \mu)$ is empty unless $\la \trianglerighteq \mu$.
We put $\CT_0(\la) = \bigcup_{\mu \in \vL}\CT_0(\la, \mu)$.
\par
For each $S \in \CT_0(\la, \mu), T \in \CT_0(\la,\nu)$, 
Dipper-James-Mathas [DJM]  
constructed an $\CH$-equivariant map 
$\vf_{ST}: M^{\nu} \to M^{\mu}$.  They showed that 
$\CS(\vL)$ is a cellular algebra, in the sense of 
Graham-Lehrer [GL] with cellular basis 
\begin{equation*}
\ZC(\vL) = \{ \vf_{ST} \mid 
S, T \in \CT_0(\la) \text{ for some } \la \in \vL^+ \}.
\end{equation*}
\para{1.4.}
Let $\Bp = (l_1, \dots, l_g)$ be a $g$-tuple of positive integers 
such that $\sum_{i=1}^g l_i = l$.
For each $\mu = (\mu^{(1)}, \dots, \mu^{(l)})$,  
one can associate a $g$-tuple of multi-compositions
$(\mu^{[1]}, \dots, \mu^{[g]})$ by making use of $\Bp$, where
\begin{equation*}
\mu^{[1]} = (\mu^{(1)}, \dots, \mu^{(l_1)}), \quad
\mu^{[2]} = (\mu^{(l_1+1)}, \dots, \mu^{(l_1+l_2)}), \quad\cdots. 
\end{equation*}
For example assume that $N = 20, l = 5$ and $\Bp = (2,2,1)$.
Take $\mu = (21; 121; 32; 1^3; 41)$.
Then $\mu$ is written as $\mu = (\mu^{[1]}, \mu^{[2]}, \mu^{[3]})$
with 
\begin{equation*}
\mu^{[1]} = (21; 121), \quad \mu^{[2]} = (32; 1^3), \quad 
\mu^{[3]} = (41).
\end{equation*}
\par
For $\mu = (\mu^{(1)}, \dots, \mu^{(l)})  
   = (\mu^{[1]}, \dots, \mu^{[g]})  \in \vL$, put
\begin{equation*}
\a_{\Bp}(\mu) = (N_1, \dots, N_g), \quad  
\Ba_{\Bp}(\mu) = (a_1, \dots, a_g),
\end{equation*}
where
$N_k = |\mu^{[k]}|$, and  
$a_k = \sum_{i=1}^{k-1}N_i$ for $k = 1, \dots, g$ with 
$a_1 = 0$.
\par
Note that we often use the following relation.
Take $\la = (\la^{[1]}, \dots, \la^{[g]}), 
\mu = (\mu^{[1]}, \dots, \mu^{[g]}) \in \vL$.  Then 
$\a_{\Bp}(\la) = \a_{\Bp}(\mu)$ if and only if 
$|\la^{[k]}| = |\mu^{[k]}|$ for $k = 1, \dots, g$.  
\para{1.5.}
Put
\begin{equation*}
\begin{split}
\ZC^{\Bp} = \{ \vf_{ST} &\in \ZC(\vL) \mid S \in \CT_0(\la, \mu),
                T \in \CT_0(\la,\nu), \\
          &\Ba_{\Bp}(\la) > \Ba_{\Bp}(\mu) \text{ if }
             \a_{\Bp}(\mu) \ne \a_{\Bp}(\nu), 
                      \mu,\nu \in \vL, \la \in \vL^+\},
\end{split}
\end{equation*}
where $\Ba_{\Bp}(\la) = (a_1, \dots, a_g) \ge 
        \Ba_{\Bp}(\mu) = (b_1, \dots, b_g)$ if 
$a_k \ge b_k$ for $k = 1, \dots, g$, and 
$\Ba_{\Bp}(\la) > \Ba_{\Bp}(\mu)$ if 
$\Ba_{\Bp}(\la) \ge \Ba_{\Bp}(\mu)$ and 
$\Ba_{\Bp}(\la) \ne \Ba_{\Bp}(\mu)$.
Let $\CS^{\Bp}$ be the $R$-span of $\ZC^{\Bp}$.  Then
by using the cellular structure,  one can show that 
$\CS^{\Bp}$ is a subalgebra of $\CS(\vL)$ containing 
the identity in $\CS(\vL)$.  The algebra $\CS^{\Bp}$
turns out to be a standardly based algebra in the 
sense of Du-Rui [DR] with respect to the poset $\vS^{\Bp}$, 
where $\vS^{\Bp}$ is a subset of 
$\vL^+\times \{ 0,1\}$ given by 
\begin{equation*}
\begin{split}
\vS^{\Bp} = (\vL^+\times \{ 0,1\}) \backslash 
            &\{ (\la, 1) \mid \CT_0(\la,\mu) = \emptyset \\
               &\text{ for any } \mu \in \vL 
               \text{ such that } \Ba_{\Bp}(\la) > \Ba_{\Bp}(\mu)\},
\end{split}
\end{equation*}
and the partial order $\ge$ on $\vS^{\Bp}$ is defined as 
$(\la_1,\ve_1) > (\la_2,\ve_2)$ if $\la_1 \triangleright \la_2$ 
or $\la_1 = \la_2$ and $\ve_1 > \ve_2$.
\para{1.6.} 
For each $\la \in \vL^+, \mu \in \vL$, we define a set 
$\CT_0^{\Bp}(\la,\mu)$ by 
\begin{equation*}
\CT_0^{\Bp}(\la,\mu) = \begin{cases}
           \CT_0(\la,\mu) &\quad\text{ if } 
              \Ba_{\Bp}(\la) = \Ba_{\Bp}(\mu), \\
           \emptyset &\quad\text{ otherwise, }
                        \end{cases}
\end{equation*}
and put 
$\CT_0^{\Bp}(\la) = \bigcup_{\mu \in \vL}\CT_0^{\Bp}(\la,\mu)$.
\par
Let $\wh\CS^{\Bp}$ be the $R$-submodule of $\CS^{\Bp}$
spanned by 
\begin{equation*}
\wh\ZC^{\Bp} = \ZC^{\Bp} \backslash 
                       \{ \vf_{ST} \mid S, T \in \CT_0^{\Bp}(\la)
                            \text{ for some }\la \in \vL^+ \}.
\end{equation*}
Then $\wh\CS^{\Bp}$ turns out to be a two-sided ideal of 
$\CS^{\Bp}$.  We denote by $\ol\CS^{\Bp} = \ol\CS^{\Bp}(\vL)$ the 
quotient algebra $\CS^{\Bp}/\wh\CS^{\Bp}$.
Let $\pi: \CS^{\Bp} \to \ol\CS^{\Bp}$ be the
natural projection, and put $\ol\vf = \pi(\vf)$ for 
$\vf \in \CS^{\Bp}$.  One can show that $\ol\CS^{\Bp}$
is a cellular algebra with cellular basis 
\begin{equation*}
\ol\ZC^{\Bp} = \{ \ol\vf_{ST} \mid S, T \in \CT_0^{\Bp}(\la)
                     \text{ for } \la \in \vL^+\}.
\end{equation*}
\par
Thus we have constructed, for each $\Bp$, a subalgebra $\CS^{\Bp}$
of $\CS(\vL)$ and its quotient algebra $\ol\CS^{\Bp}$. 
So we are in the following situation,
\begin{equation*}
\begin{CD}
\CS^{\Bp}  @>\io>> \CS(\vL) \\
@V\pi VV            @.    \\
\ol\CS^{\Bp}    @.
\end{CD}
\end{equation*}
where $\io$ is an injection and $\pi$ is a surjection.
\remark{1.7.}
In the special case where $\Bp = (1,\dots,1)$, 
we have $g = l$, and $\mu^{[k]} = \mu^{(k)}$ for 
$k = 1, \dots, g = l$.  Moreover in this case 
$\CS^{\Bp}$ and $\ol\CS^{\Bp}$ coincide with 
the subalgebra $\CS^0(\Bm, N)$ and its quotient 
$\CS(\Bm,N)$ considered in [SawS] in connection with 
the Schur-Weyl duality between $\CH$ and the quantum 
group $U_v(\Fg)$, where 
$\Fg = \Fg\Fl_{m_1}\oplus\cdots\oplus \Fg\Fl_{m_l}$
(cf. [SakS], [HS]).  The other extreme case is 
$\Bp = (l)$, in which case $\CS^{\Bp}$ and $\ol\CS^{\Bp}$
coincide with $\CS(\vL)$.  Thus in general $\CS^{\Bp}$
are regarded as intermediate objects between $\CS(\vL)$ and
$\CS^0(\Bm, N)$.  
\para{1.8.}
In the rest of this section, unless otherwise stated  
we assume that $R$ is a field.  
By a general theory of cellular algebras, one can define, 
for each $\la \in \vL^+$,  
the  Weyl module $W^{\la}$ and its irreducible quotient 
$L^{\la}$ for $\CS(\vL)$.  Similarly, since $\ol\CS^{\Bp}$
is a cellular algebra, we have the Weyl module $\ol Z^{\la}$
and its irreducible quotient $\ol L^{\la}$. Note that 
$\CS(\vL)$ (resp. $\ol \CS^{\Bp}$) is a quasi-hereditary algebra, 
and so the set $\{ L^{\la} \mid \la \in \vL^+\}$ 
(resp. $\{ \ol L^{\la} \mid \la \in \vL^+\}$) gives 
a complete set of representatives of irreducible $\CS(\vL)$-modules
(resp. irreducible $\ol\CS^{\Bp}$-modules). 
\par
On the other hand, by using the general theory of standardly based
algebras, one can construct, for each $\e = (\la,\ve) \in \vS^{\Bp}$, 
the Weyl module $Z^{\e}$, and its irreducible quotient $L^{\e}$ 
(if it is non-zero).  Thus the set 
$\{ L^{\e} \mid \e \in \vS^{\Bp}, L^{\e} \ne 0\}$ gives a complete
set of representatives of irreducible $\CS^{\Bp}$-modules.  
In the case where $\e = (\la, 0)$, we know more; 
$L^{(\la,0)}$ is always non-zero for $\la \in \vL^+$, and
the composition factors of $Z^{(\la,0)}$ are all of the form 
$L^{(\mu, 0)}$ for some $\mu \in \vL^+$.
We shall discuss the relations
among the decomposition numbers
\begin{equation*}
[W^{\la} : L^{\mu}]_{\CS(\vL)}, \quad  
[Z^{(\la,0)} : L^{(\mu,0)}]_{\CS^{\Bp}}, \quad
[\ol Z^{\la} : \ol L^{\mu}]_{\ol\CS^{\Bp}}.
\end{equation*}
(In order to distinguish the decomposition numbers for 
$\CS(\vL), \CS^{\Bp}$ and $\ol\CS^{\Bp}$, we use the subscripts
such as $[W^{\la} :  L^{\mu}]_{\CS(\vL)}$).
In the case where $\Bp = (1,\dots, 1)$, Sawada [Sa] discussed 
these relations. Our result below is a generalization of his result
for the general $\Bp$. 
\para{1.9.}
First we consider the relation between the decomposition numbers 
of $\CS^{\Bp}$ and that of $\ol\CS^{\Bp}$.  
Under the surjection $\pi: \CS^{\Bp} \to \ol\CS^{\Bp}$,
we regard an $\ol\CS^{\Bp}$-module as an $\CS^{\Bp}$-module.
We have the following lemma.
\begin{lem}  
For $\la, \mu \in \vL^+$, we have 
\begin{enumerate}
\item
$\ol Z^{\la} \simeq Z^{(\la,0)}$ as $\CS^{\Bp}$-modules.
\item
$\ol L^{\mu} \simeq L^{(\mu,0)}$ as $\CS^{\Bp}$-modules.
\item
$[\ol Z^{\la} :  \ol L^{\mu}]_{\ol\CS^{\Bp}} = 
            [Z^{(\la,0)} : L^{(\mu,0)}]_{\CS^{\Bp}}$.
\item
Assume that $\a_{\Bp}(\la) \ne \a_{\Bp}(\mu)$.
Then we have $[\ol Z^{\la} :  \ol L^{\mu}]_{\ol\CS^{\Bp}} = 0$.
\end{enumerate}
\end{lem}
\para{1.11.}
Next we consider the relationship between the decomposition 
numbers of $\CS^{\Bp}$ and that of $\CS(\vL)$.  
Under the injection $\io: \CS^{\Bp} \hookrightarrow \CS(\vL)$,
we regard an $\CS(\vL)$-module as an $\CS^{\Bp}$-module.
The following proposition was first proved in [Sa] in the case
where $\Bp = (1,1,\dots, 1)$.  A similar argument works also for
a general $\Bp$.
\begin{prop} 
For each $\la \in \vL^+$, there exists an isomorphism of
$\CS(\vL)$-modules
\begin{equation*}
Z^{(\la,0)}\otimes_{\CS^{\Bp}}\CS(\vL) \simeq W^{\la}.
\end{equation*}
\end{prop}
\par
By using Lemma 1.10 and Proposition 1.12, we have the 
following theorem, which is a generalization of 
[Sa, Th. 5.7]. In fact in the theorem, the inequality 
\begin{equation*}
[Z^{(\la,0)}: L^{(\mu,0)}]_{\CS^{\Bp}} 
              \le [W^{\la} : L^{\mu}]_{\CS(\vL)}
\end{equation*}
always holds, and the converse inequality holds only 
when $\a_{\Bp}(\la) = \a_{\Bp}(\mu)$. 

\begin{thm}[{[SW, Th. 3.13]}]  
For any $\la, \mu \in \vL^+$
such that $\a_{\Bp}(\la) = \a_{\Bp}(\mu)$, 
we have
\begin{equation*}
[\ol Z^{\la} : \ol L^{\mu}]_{\ol\CS^{\Bp}}
= [Z^{(\la,0)} : L^{(\mu,0)}]_{\CS^{\Bp}}
= [W^{\la} : L^{\mu}]_{\CS(\vL)}.
\end{equation*}
\end{thm}
\para{1.14.}
In view of Theorem 1.13, the determination of the decomposition numbers 
$[W^{\la} : L^{\mu}]$ is reduced to that of the decomposition 
numbers $[\ol Z^{\la} : \ol L^{\mu}]_{\ol\CS^{\Bp}}$ for $\ol\CS^{\Bp}$
as far as the case where $\a_{\Bp}(\la) = \a_{\Bp}(\mu)$.
The algebra $\ol\CS^{\Bp}$ has a remarkable structure 
as the following formula shows.  In order to state our result,
we prepare some notation.
For each $N_k \in \ZZ_{\ge 0}$, put 
$\vL_{N_k} = \vL_{N_k,l_k}(\Bm^{[k]})$ and 
$\vL^+_{N_k} = \vL^+_{N_k,l_k}(\Bm^{[k]})$.
($\vL_{N_k}$ or $\vL_{N_k^+}$ is regarded as the empty set if 
$N_k = 0$.)
For each $\mu^{[k]} \in \vL_{N_k}$, the $\CH_{N_k,l_k}$-module
$M^{\mu^{[k]}}$ is defined as in the case of the 
$\CH$-module $M^{\mu}$,
and the cyclotomic $q$-Schur algebra $\CS(\vL_{N_k})$
associated to the Ariki-Koike algebra $\CH_{N_k,l_k}$ is defined.
The following theorem was first proved in [SawS] for 
$\Bp = (1,1,\dots,1)$ under a certain condition on parameters.
Here we don't need any assumption on parameters.
\begin{thm}[{[SW, Th. 4.15]}]  
There exists an isomorphism of $R$-algebras
\begin{equation*}
\ol\CS^{\Bp} \simeq \bigoplus_{\substack{ (N_1, \dots, N_g) \\
     N_1 + \cdots + N_g = N}}\CS(\vL_{N_1})
               \otimes\cdots\otimes\CS(\vL_{N_g}).
\end{equation*}
\end{thm}
For $\la^{[k]}, \mu^{[k]} \in \vL_{N_k}^+$, let 
$W^{\la^{[k]}}$ be the Weyl module, and $L^{\mu^{[k]}}$ be the 
irreducible module with respect to $\CS(\vL_{N_k})$.  As a corollary 
to the theorem, we have
\begin{cor}  
Let $\la, \mu \in \vL^+$.  
Then under the isomorphism in Theorem 1.15, we have the following.
\begin{enumerate}
\item
$\ol Z^{\la} \simeq W^{\la^{[1]}}\otimes\cdots\otimes
     W^{\la^{[g]}}$.
\item
$\ol L^{\mu} \simeq L^{\mu^{[1]}}\otimes\cdots\otimes 
L^{\mu^{[g]}}$.
\item
$[\ol Z^{\la} : \ol L^{\mu}]_{\ol\CS^{\Bp}} = \begin{cases}
     \prod_{i=1}^g[ W^{\la^{[i]}}: L^{\mu^{[i]}}]_{\CS(\vL_{N_i})}
                   &\quad\text{ if } \a_{\Bp}(\la) = \a_{\Bp}(\mu), \\
     0             &\quad\text{ otherwise. }
                           \end{cases}$
\end{enumerate}
\end{cor}
Combining this with Theorem 1.13, we have the following product formula
for the decomposition numbers of $\CS(\vL)$.  The special case where 
$\Bp = (1,\dots, 1)$ is due to [Sa, Cor. 5.10], (still under a certain 
condition on parameters).
\begin{thm}[{[SW, Theorem 4.17]}]  
For $\la, \mu \in \vL^+$ such that 
$\a_{\Bp}(\la) = \a_{\Bp}(\mu)$, we have
\begin{equation*}
[W^{\la} : L^{\mu}]_{\CS(\vL)} 
     = \prod_{i=1}^g[W^{\la^{[i]}} : L^{\mu^{[i]}}]_{\CS(\vL_{N_i})}.
\end{equation*}
\end{thm}
\para{1.18}
By making use of the Jantzen filtration, we shall define 
a polynomial analogue of the decomposition numbers, namely
for each $\la, \mu \in \vL^+$, we define a polynomial 
$d_{\la\mu}(q) \in \ZZ[q]$ with indeterminate $q$ such that
$d_{\la\mu}(1)$ coincides with the decomposition number 
$[W^{\la}: L^{\mu}]_{\CS(\vL)}$.  
We define similar polynomials
also in the case for $\CS^{\Bp}$ and $\ol\CS^{\Bp}$.
\par
We assume that $R$ is a discrete valuation ring
with the maximal ideal $\Fp$, and let $F = R/\Fp$ be the
quotient field. 
We fix parameters $\wh v, \wh Q_1, \dots, \wh Q_l$ in $R$, and
let $v, Q_1, \dots, Q_l \in F$ be their images under the natural 
map $R \to R/\Fp = F$.
Let $\CS_R = \CS_R(\vL)$ be the cyclotomic $\wh v$-Schur algebra
over $R$ 
with parameters $\wh v, \wh Q_1, \dots, \wh Q_l$, and 
$\CS = \CS(\vL)$ be the cyclotomic $v$-Schur algebra 
over $F$ with parameters $v, Q_1, \dots, Q_l$.  Thus
$\CS \simeq (\CS_R + \Fp\CS_R)/\Fp\CS_R$.
The algebras $\CS^{\Bp}_R, \ol\CS^{\Bp}_R$ over $R$, and
the algebras $\CS^{\Bp}, \ol\CS^{\Bp}$ over $F$ are defined as before.
Let $W_R^{\la}$ be the Weyl module of $\CS_R$, and let 
$\lp \ , \ \rp$ be the canonical bilinear form on $W_R^{\la}$
arising from the cellular structure of $\CS(\vL)_R$.
For $i = 0,1, \dots, $ put
\begin{equation*}
W_R^{\la}(i) = \{ x \in W_R^{\la} \mid \lp x, y\rp \in \Fp^i 
                    \text{ for any } y \in W_R^{\la}\}
\end{equation*}
and define an $F$-vector space 
\begin{equation*}
W^{\la}(i) = (W_R^{\la}(i) + \Fp W_R^{\la})/\Fp W_R^{\la}.
\end{equation*}
Then $W^{\la}(0) = W^{\la}$ is the Weyl module of $\CS$, 
and we have a filtration 
\begin{equation*}
W^{\la} = W^{\la}(0) \supset W^{\la}(1) \supset 
                   W^{\la}(2) \supset \cdots 
\end{equation*}
of $W^{\la}$, which is the Jantzen filtration of $W^{\la}$.
\par
Similarly, by using the cellular structure of $\ol\CS^{\Bp}$, 
and by using the property of the standardly based algebra of 
$\CS^{\Bp}$, one can define the Jantzen filtrations,
\begin{align*}
\ol Z^{\la} &= \ol Z^{\la}(0) \supset \ol Z^{\la}(1) \supset
                   \ol Z^{\la}(2) \supset\cdots, \\
Z^{(\la,0)} &= Z^{(\la,0)}(0) \supset Z^{(\la,0)}(1) 
                 \supset Z^{(\la,0)}(2) \supset\cdots.
\end{align*}
Since $W^{\la}$ (resp. $Z^{(\la,0)}, \ol Z^{\la}$) 
is a finite dimensional $F$-vector space, the Jantzen filtration 
gives a finite sequence. 
One sees that $W^{\la}(i)$ is an $\CS$-submodule 
of $W^{\la}$ by the associativity of the bilinear form, and
similarly for $Z^{(\la,0)}$ and $\ol Z^{\la}$. 
Thus we define a polynomial $d_{\la\mu}(q)$ by 
\begin{equation*}
d_{\la\mu}(v) = \sum_{i \ge 0}
   [W^{\la}(i)/W^{\la}(i+1) : L^{\mu}]q^i,
\end{equation*}
where $[M : L^{\mu}] = [M : L^{\mu}]_{\CS}$ denotes the multiplicity 
of $L^{\mu}$ in the composition series of the $\CS$-module $M$
as before. (In the notation below, we omit the subscripts $\CS$, etc.)
Similarly, we define, for $Z^{(\la,0)}$ and $\ol Z^{\la}$, 
\begin{align*}
d_{\la\mu}^{(\la,0)}(q) &= \sum_{i \ge 0}
    [Z^{(\la,0)}(i)/Z^{(\la,0)}(i+1) : L^{(\mu,0)}]q^i, \\
\ol d_{\la\mu}(q) &= \sum_{i \ge 0}
    [\ol Z^{\la}(i)/\ol Z^{\la}(i+1) : \ol L^{\mu}]q^i.
\end{align*}
$d_{\la\mu}(q), d_{\la\mu}^{(\la,0)}(q)$ and $\ol d_{\la\mu}(q)$ are
polynomials in $\ZZ_{\ge 0}[q]$ and we call them $q$-decomposition 
numbers. Note that 
$d_{\la\mu}(1)$ coincides with $[W^{\la} : L^{\mu}]$, 
and similarly, we have 
$d_{\la\mu}^{(\la,0)}(1) = [Z^{(\la,0)} :  L^{(\la,0)}]$, 
$\ol d_{\la\mu}(1) = [\ol Z^{\la} : \ol L^{\mu}]$.
\par
As a $q$-analogue of Theorem 1.13 and Theorem 1.17, we have the
following product formula for $q$-decomposition numbers.
\begin{thm}[{[W, Th. 2.8, Th. 2.14]}]
For $\la,\mu \in \vL^+$ such that $\a_{\Bp}(\la) = \a_{\Bp}(\mu)$,
we have 
\begin{equation*}
d_{\la\mu}(q) = \ol d_{\la\mu}(q) = 
       \prod_{i=1}^g d_{\la^{[i]}\mu^{[i]}}(q).
\end{equation*}
\end{thm}
\section{Product formula for the canonical bases of the Fock space}
\para{2.1.}
In the remainder of this paper, we basically follow the notation in Uglov [U]. 
First we review some notations. 
Fix positive integers $n$, $l$. 
Let $\Pi^l= \{\la=(\la^{(1)},\cdots,\la^{(l)}) \}$ be 
the set of $l$-partitions.  
Take an $l$-tuple $\Bs=(s_1,\cdots,s_l)\in \ZZ^l$, 
and  call it a multi-charge. 
Let $U_q(\wh{\Fs\Fl}_n)$ be the quantum group of 
type $A_{n-1}^{(1)}$.
The $q$-deformed Fock space $\BF_q[\Bs]$ of level $l$ with 
multi-charge $\Bs$ is defined as 
a vector space over $\QQ(q)$ with a basis
$\{|\la,\Bs \rp \mid \la \in \Pi^l \}$, equipped with 
a $U_q(\wh{\Fs\Fl}_n)$-module structure. 
The $U_q(\wh{\Fs\Fl}_n)$-module structure is defined as in 
[U, Th. 2.1], which depends on the choice of $\Bs$. 
\para{2.2.}
Put $s=s_1+\cdots +s_l$ for a multi-charge $\Bs = (s_1,\cdots, s_l)$. 
Let $\BP(s)$ be the set of semi-infinite sequences 
$\Bk=(k_1,k_2,\cdots )\in \ZZ^{\infty}$ such that 
$k_i=s-i+1$ for all $i\gg 1$, 
and $\BP^{++}(s)= \{\Bk=(k_1,k_2,\cdots)\in \BP(s) \mid  k_1>k_2>\cdots \}$. 
For $\Bk\in \BP(s)$, put $u_{\Bk}=u_{k_1}\we u_{k_2} \we \cdots $, and  
call it a semi-infinite wedge.  In the case where
$\Bk \in \BP^{++}(s)$ we call it an ordered semi-infinite wedge. 
\par
Let $\vL^{s+\frac{\infty}{2}}$ be a vector space over $\QQ(q)$ spanned
by $\{u_{\Bk} \mid \Bk\in \BP(s) \}$ satisfying the ordering rule 
[U, Prop. 3.16].  
By the ordering rule 
any semi-infinite wedge $u_{\Bk}$ can be written as a linear combination 
of some ordered semi-infinite wedges. 
It is known (cf. [U, Prop. 4.1]) that 
$\vL^{s+\frac{\infty}{2}}$ has a basis $\{u_{\Bk} \mid
\Bk \in \BP^{++}(s) \}$. 
 \par
The vector space $\vL^{s+\frac{\infty}{2}}$ is called a semi-infinite 
wedge product. 
By [U, 4.2], $\vL^{s+\frac{\infty}{2}}$ has a structure of 
a $U_q(\wh{\Fs\Fl}_n)$-module. 
Let $\ZZ^{l}(s)= \{\Bs=(s_1,\cdots ,s_l)\in \ZZ^{l} \mid
s = \sum s_i \}$. 
Then we have 
\begin{equation*} 
\tag{2.2.1}
\vL^{s+\frac{\infty}{2}}\simeq \bigoplus_{\Bs\in \ZZ^{l}(s)}\BF_q[\Bs] 
\qquad \text{as }U_q(\wh{\Fsl}_n) \text{-modules.}
\end{equation*}
Thus we can regard $\BF_q[\Bs]$ as a $U_q(\wh{\Fs\Fl}_n)$-submodule 
of $\vL^{s+\frac{\infty}{2}}$. 
The isomorphism in (2.2.1) is given through a bijection between 
two basis $\{u_{\Bk}  \mid \Bk\in \BP^{++}(s) \}$ 
and $\{|\la,\Bs \rp \mid \la \in \Pi^l,\,\Bs\in \ZZ^l(s) \}$ 
as in [U, 4.1]. 
Identifying these bases, we write $u_{\Bk} = |\la,\Bs\rp$ 
if $|\la,\Bs\rp$ corresponds
to $u_{\Bk}$.
\para{2.3.}
For later use,  we explain the explicit correspondence 
$u_{\Bk} \lra |\la, \Bs\rp$ given in [U, 4.1].
Assume given $u_{\Bk}$. Then for each $i \in \ZZ_{\ge 1}$, 
$k_i$ is written as $k_i = a_i + n(b_i-1) - nlm_i$, where 
$a_i \in \{ 1, \dots, n\}$, $b_i \in \{1,\dots, l\}$ and 
$m_i \in \ZZ$ are determined uniquely.
For $b \in \{ 1, \dots, l\}$, let $k_1^{(b)}$ be equal 
$a_i -nm_i$ where $i$ is the smallest number such that $b_i = b$, 
and let $k_2^{(b)}$ be equal $a_j - nm_j$ where $j$ is the next 
smallest number such that $b_j = b$, and so on.
In this way, we obtain a strictly decreasing sequence 
$\Bk^{(b)} = (k_1^{(b)}, k_2^{(b)}, \dots)$ 
such that $k_i^{(b)} = s_b - i + 1$
for $i \gg 1$ for some uniquely determined integer $s_b$.
Thus $\Bk^{(b)} \in \BP^{++}(s_b)$, and one can define  a partition
$\la^{(b)} = (\la_1^{(b)}, \la_2^{(b)}, \dots)$ by 
$\la_i^{(b)} = k_i^{(b)} - s_b + i-1$.
We see that $\sum_bs_b = s$, and  we obtain 
$\la = (\la^{(1)}, \dots, \la^{(l)})$ and $\Bs = (s_1, \dots, s_l)$.
$u_{\Bk} \to |\la, \Bs\rp$ gives the required bijection.
\par
Note that the correspondence $u_{\Bk^{(b)}} \lra |\la^{(b)}, s_b\rp$ 
for each $b$ is nothing but the correspondence 
$\vL^{s_b+\frac{\infty}{2}} \simeq \BF_q[s_b]$ 
in the case where $\BF_q[s_b]$ is a level 1 Fock space with charge
$s_b$.
\para{2.4.}
In [U], Uglov defined a bar-involution 
$\rule[5pt]{10pt}{0.5pt}$ on $\vL^{s+\frac{\infty}{2}}$ 
by making use of the 
realization of the semi-infinite wedge product in terms of 
the affine Hecke algebra, which is semi-linear with respect to the 
involution $q \mapsto q\iv$ on $\QQ(q)$, and commutes with the action of
$U_q(\wh{\Fs\Fl}_n)$, i.e., $\ol{u\cdot x} = \ol u\cdot \ol x$ for 
$u \in U_q(\wh \Fsl_n), x \in \vL^{s+\frac{\infty}{2}}$ 
(here $\ol u$ is the usual bar-involution on $U_q(\wh \Fsl_n)$ ).
We give a property of the bar-involution on $\vL^{s+\frac{\infty}{2}}$,
which makes it possible to compute explicitly the bar-involution.
\par
For $\Bk\in \BP^{++}(s)$, we have 
\begin{equation*}
\tag{2.4.1}
\ol{u_{\Bk}}= \ol{u_{k_1}\we u_{k_2}\we \cdots \we u_{k_r}}\we
 u_{k_{r+1}} \we u_{k_{r+2}} 
   \we\cdots 
\end{equation*}
for any $r \gg 1$.  Moreover, for any $(k_1, k_2, \dots, k_r)$,
not necessarily ordered, we have 
\begin{equation*}
\tag{2.4.2}
\ol{u_{k_1}\we u_{k_2} \we\cdots \we u_{k_r}}  =\a(q) \,u_{k_r}\we
 \cdots \we u_{k_2} \we u_{k_1} 
\end{equation*}
with some $\a(q) \in \QQ(q)$ of the form $\pm q^a$.  
The quantity $\a(q)$ is given explicitly 
as in  [U, Prop. 3.23]. Thus one can express
$\ol u_k$ by the ordering rule as a linear combination of 
ordered semi-infinite wedges.  
\par
The bar-involution on $\vL^{s+\frac{\infty}{2}}$ leaves the subspace
$\BF_q[\Bs]$ invariant, and so defines a bar-involution on the 
Fock space $\BF_q[\Bs]$.  Let us define $\CL^+$ 
(resp. $\CL^-$) as the $\QQ[q]$-lattice (resp. $\QQ[q\iv]$-lattice) of 
$\vL^{s+\frac{\infty}{2}}$ generated by 
$\{ |\la, \Bs\rp \mid \la \in \Pi^l, \Bs \in \ZZ^l(s)\}$.
Under this setting, Uglov constructed the canonical bases on 
$\BF_q[\Bs]$.
\begin{prop}[{[U, Prop. 4.11] }]  
There exist unique bases $\{ \CG^+(\la,\Bs)\}, \{ \CG^-(\la,\Bs)\}$ 
of $\BF_q[\Bs]$ satisfying the following properties;
\begin{enumerate}
\item $\ol{\CG^+(\la,\Bs)}=\CG^+(\la,\Bs)$, \qquad 
        $\ol{\CG^-(\la,\Bs)}=\CG^-(\la,\Bs)$, 
\item $\CG^+(\la,\Bs) \equiv |\la, \Bs \rp \mod q\CL^+$, \qquad 
         $\CG^-(\la,\Bs) \equiv |\la, \Bs \rp \mod q^{-1}\CL^-$,
\end{enumerate}
\end{prop}
\para{2.6.}
We define $\vD^{\pm}_{\la, \mu}(q) \in \QQ[q^{\pm 1}]$, 
for $\la, \mu \in \Pi^l$,  by the formula
\begin{equation*}
\CG^{\pm}(\la,\Bs) = \sum_{\mu \in \Pi^l} 
       \vD^{\pm}_{\la,\mu}(q)\, |\mu,\Bs \rp. 
\end{equation*}
Note that $\vD_{\la,\mu}^{\pm}(q) =  0$  
unless $|\la| = |\mu|$.                 
\par 
For $\la \in \Pi^l$, $\Bs = (s_1,\dots, s_l)$, 
and $M \in \ZZ$, we say that $|\la, \Bs\rp$ is $M$-dominant if
$s_i - s_{i+1} > M + |\la |$ for $i = 1, \dots, l-1$.
\par
\remark{2.7.} Let $\CS(\vL)$ be the cyclotomic $v$-Schur algebra
over $R$ with parameters $v, Q_1, \cdots, Q_l$.  We consider 
the special setting for parameters as follows;
$R = \CC$ and 
$(v; Q_1, \dots, Q_l) = (\xi; \xi_1^{s_1}, \dots, \xi^{s_l})$,
where $\xi = \exp(2\pi i/n) \in \CC$ and $\Bs = (s_1, \dots, s_l)$
is a multi-charge.   
For $\la = (\la^{(1)}, \dots, \la^{(l)}) \in \Pi^l$, 
we define an $l$-partition $\la^{\dag}$ by 
\begin{equation*}
\la^{\dag} = ((\la^{(l)})', (\la^{(l-1)})', \dots, (\la^{(1)})')),
\end{equation*}
where $(\la^{(i)})'$ denotes the dual partition of the partition 
$\la^{(i)}$.
Recall that $d_{\la\mu}(q) \in \ZZ[q]$ 
is the $v$-decomposition 
number defined in 1.18. 
In [Y], Yvonne gave the following conjecture;
\par\medskip\noindent
{\bf Conjecture I:} \ Assume that $|\la, \Bs\rp$ is 0-dominant.
Then we have 
\begin{equation*}
d_{\la\mu}(q) = \vD^+_{\mu^{\dag}\la^{\dag}}(q).
\end{equation*}
By specializing $q = 1$, Conjecture I implies an LLT-type conjecture 
for decomposition numbers of $\CS(\vL)$, 
\par\medskip\noindent
{\bf Conjecture II: } \ Under the same setting as in Conjecture I, 
we have 
\begin{equation*}
[W^{\la} : L^{\mu}]_{\CS(\vL)} = \vD^+_{\mu^{\dag}\la^{\dag}}(1).
\end{equation*}
\par
In the case where $l = 1$, i.e., the case where 
$\CS(\vL)$ is the $v$-Schur algebra associated to the Iwahori-Hecke 
algebra of type $A$, Conjecture II was proved by 
Varagnolo-Vasserot [VV].
It is open for the general case, $l > 1$.  Concerning Conjecture I, 
it is not yet verified even in the case where $l = 1$. 
\para{2.8.}
Fix $\Bp = (l_1, \dots, l_g) \in \ZZ^g_{>0}$ such that 
$\sum_{i=1}^g l_i = l$ as in 1.4.
For $i = 1, \dots, g$, define $\Bs^{[i]}$ by 
$\Bs^{[1]} = (s_1, \dots, s_{l_1}), 
\Bs^{[2]} = (s_{l_1+1}, \dots, s_{l_1+l_2})$,  and so on. 
Thus we can write
$\Bs = (\Bs^{[1]}, \dots, \Bs^{[g]})$.  For each $\la \in \Pi^l$, 
we express it as $\la = (\la^{[1]}, \dots, \la^{[g]})$ as in 
1.4. Recall the integer $\a_{\Bp}(\la)$ in 1.4. 
We have $\a_{\Bp}(\la) = \a_{\Bp}(\mu)$ if and only if 
$|\la| = |\mu|$ and  
$|\la^{[i]}| = |\mu^{[i]}| $ for $i = 1, \dots, g$.
\par
Let $\BF_q[\Bs^{[i]}]$ be the $q$-deformed Fock space of level $l_i$ 
with multi-charge $\Bs^{[i]}$, with
basis $\{|\la^{[i]},\Bs^{[i]} \rp \mid \la^{[i]}\in \Pi^{l_i} \}$.
We consider the canonical bases 
$\{\CG^{\pm}(\la^{[i]},\Bs^{[i]}) \mid  \la^{[i]}\in \Pi^{l_i}\}$ 
of $\BF_q[\Bs^{[i]}]$. 
Put
\begin{equation*}
\CG^{\pm}(\la^{[i]},\Bs^{[i]}) = \sum_{\mu\in \Pi^{l_i}}
 \vD^{\pm}_{\la^{[i]},\mu^{[i]}}(q)
             \,|\mu^{[i]},\Bs^{[i]}\rp
\end{equation*}
with $\vD^{\pm}_{\la^{[i]},\mu^{[i]}}(q)\in \QQ[q^{\pm 1}]$. 
The following product formula is our main theorem, which is a counter-part
of Theorem 1.19 to the case of the Fock space, in view of 
Conjecture I.  
\begin{thm} 
Let $\la, \mu \in \Pi^l$ be such that $|\la, \Bs\rp$ is $M$-dominant 
for $M > 2n$, and that $\a_{\Bp}(\la) = \a_{\Bp}(\mu)$.
Then we have
\begin{equation*}
\vD_{\la,\mu}^{\pm}(q) = \prod_{i = 1}^g
             \vD^{\pm}_{\la^{[i]}, \mu^{[i]}}(q).
\end{equation*}
\end{thm}
As a corollary, we obtain a special case of Conjecture II 
(though we require a stronger dominance condition for $|\la, \Bs\rp$).
\begin{cor}
Let $\la, \mu \in \Pi^l$ be such that $|\la^{(i)}| = |\mu^{(i)}|$
for $i = 1, \dots, l$. Assume that $|\la,\Bs\rp$ is 
$M$-dominant for $M > 2n$. Then we have
\begin{equation*}
[W^{\la} : L^{\mu}]_{\CS(\vL)} = 
              \vD_{\mu^{\dag}\la^{\dag}}^{+}(1).
\end{equation*}
\end{cor}
\begin{proof}
Take $\Bp = (1,\dots, 1)$.  Then we have 
$\la^{[i]} = \la^{(i)}$ for $i = 1, \dots, g = l$, and
$\CS(\vL_{N_i})$ coincides with the $v$-Schur algebra of type $A$.
By applying Theorem 1.17 or Theorem 1.19, we have
\begin{equation*}
[W^{\la} : L^{\mu}]_{\CS(\vL)} = \prod_{i=1}^l
              [W^{\la^{(i)}}: L^{\mu^{(i)}}]_{\CS(\vL_{N_i})}.
\end{equation*}
Also, by applying Theorem 2.9, we have
\begin{equation*}
\vD^+_{\mu^{\dag}\la^{\dag}}(1) 
    = \prod_{i=1}^l\vD^+_{(\mu^{(i)})'(\la^{(i)})'}(1). 
\end{equation*}
On the other hand, we know 
$[W^{(\la^{(i)})} : L^{\mu^{(i)}}]_{\CS(\vL_{N_i})} = 
   \vD^+_{(\mu^{(i)})'(\la^{(i)})'}(1)$
by a result of Valagnolo-Vasserot (see Remark 2.7). 
The corollary follows from these formulas.
\end{proof}
\para{2.11}
Clearly the proof of the theorem is reduced to the case where $g = 2$, i.e., 
the case where $\la = (\la^{[1]}, \la^{[2]})$, etc.  So, we assume
that $\Bp = (l_1, l_2) = (t, l-t)$ for some $t \in \ZZ_{>0}$.
We write the multi-charge $\Bs$ as $\Bs = (\Bs^{[1]}, \Bs^{[2]})$, and
consider the $q$-deformed Fock spaces $\BF_q[\Bs^{[i]}]$ of level
$l_i$ with multi-charge $\Bs^{[i]}$ for $i = 1, 2$.   
We have an isomorphism 
$\BF_q[\Bs] \simeq \BF_q[\Bs^{[1]}]\otimes \BF_q[\Bs^{[2]}]$ 
of vector spaces via the bijection of the bases  
$|\la,\Fs\rp \lra |\la^{[1]},\Bs^{[1]}\rp
               \otimes |\la^{[2]},\Bs^{[2]}\rp$ for
each $\la = (\la^{[1]},\la^{[2]}) \in \Pi^l$.
For $\Bs \in \ZZ^l(s)$, put $s' = s_1 + \cdots + s_t$, 
$s'' = s_{t+1}+\cdots + s_l$. 
Under the isomorphism in (2.2.1) we have
\begin{equation*}
\vL^{s'+\frac{\infty}{2}} \simeq \bigoplus_{\Bs^{[1]}\in \ZZ^{l_1}(s')}
    \BF_q[\Bs^{[1]}], \qquad 
\vL^{s''+\frac{\infty}{2}}\simeq \bigoplus_{\Bs^{[2]}\in \ZZ^{l_2}(s'')}
\BF_q[\Bs^{[2]}].
\end{equation*}
Then we have an injective $\QQ(q)$-linear map 
\begin{equation*}
\tag{2.11.1}
\vL^{s'+\frac{\infty}{2}}\otimes \vL^{s''+\frac{\infty}{2}} \simeq 
	\bigoplus_{\Bs^{[1]}\in \ZZ^{l_1}(s') 
   \atop \Bs^{[2]}\in \ZZ^{l_2}(s'')}
	 \BF_q[\Bs^{[1]}]\otimes \BF_q[\Bs^{[2]}]
\rightarrow  \bigoplus_{\Bs\in \ZZ^l(s)} \BF_q[\Bs] 
          \simeq \vL^{s +\frac{\infty}{2}}
\end{equation*}
via $|\la^{[1]},\Bs^{[1]}\rp \otimes |\la^{[2]},\Bs^{[2]}\rp 
         \mapsto |\la,\Bs\rp$.
We denote the embedding in (2.11.1) by $\F$.
\para{2.12.} For $\la, \mu \in \Pi^l$, we define 
$\Ba(\la) > \Ba(\mu)$ if $|\la| = |\mu|$ and 
$|\la^{[1]}| > |\mu^{[1]}|$.  Note that this is the same
as the partial order $\Ba_{\Bp}(\la) > \Ba_{\Bp}(\mu)$ 
defined in 1.5 for the case 
where $\Bp = (l_1, l_2)$.  We have the following proposition.
\begin{prop} 
Assume that $u_{\Bk}=|\la,\Bs\rp$ is $M$-dominant for 
$M > 2n$. Under the embedding 
$\F: \vL^{s'+\hinf}\otimes \vL^{s''+\hinf} \to \vL^{s + \hinf}$ 
in 2.11, we have 
\begin{equation*}
\ol{|\la,\Bs \rp}= \ol{|\la^{[1]},\Bs^{[1]}\rp }\otimes 
     \ol{|\la^{[2]},\Bs^{[2]}\rp}\ 
	+\sum_{\mu\in \Pi^l \atop \Ba(\la) > \Ba(\mu) }
               \a_{\la,\mu}\,|\mu, \Bs\rp
\end{equation*}
with $\a_{\la,\mu}\in \QQ[q,q\iv]$.
\end{prop}
\para{2.14.}
Proposition 2.13 will be proved in 3.11 in the next section.
Here assuming the proposition, we continue the proof of the 
theorem. 
We have the following result.
\begin{thm} 
Assume that $|\la,\Bs\rp = |\la^{[1]},\Bs^{[1]}\rp \otimes 
  |\la^{[2]},\Bs^{[2]}\rp$ is $M$-dominant for $M > 2n$.  
Then we have 
\begin{equation*}
\begin{split}
\CG^{\pm}(\la,\Bs) = \CG^{\pm}(\la^{[1]},&\Bs^{[1]}) \otimes 
      \CG^{\pm}(\la^{[2]},\Bs^{[2]})  \\
&+ \sum_{\mu \in \Pi^l \atop \Ba(\la) >  \Ba(\mu)}b^{\pm}_{\la,\mu}
   \,\CG^{\pm}(\mu^{[1]},\Bs^{[1]})\otimes
 \CG^{\pm}(\mu^{[2]},\Bs^{[2]})
\end{split}
\end{equation*}
with $b^{\pm}_{\la,\mu}\in \QQ[q^{\pm 1}]$. 
\end{thm}

\begin{proof}
Throughout the proof, we write 
$\D^{\pm}_{\la^{[i]},\mu^{[i]}}(q)$ as 
$\D^{\pm}_{\la^{[i]},\mu^{[i]}}$ for simplicity.
Since
\begin{align*}
& \CG^{\pm}(\la^{[1]},\Bs^{[1]})\otimes 
   \CG^{\pm}(\la^{[2]},\Bs^{[2]}) \\
&= \sum_{\mu^{[1]}\in \Pi^{l_1}}
  \D_{\la^{[1]},\mu^{[1]}}^{\pm} |\mu^{[1]},\Bs^{[1]}\rp 
	\otimes \sum_{\mu^{[2]}\in \Pi^{l_2}}
   \D_{\la^{[2]},\mu^{[2]}}^{\pm} |\mu^{[2]},\Bs^{[2]}\rp\\
&= \sum_{\mu\in \Pi^l}\D^{\pm}_{\la^{[1]},\mu^{[1]}}
    \D^{\pm}_{\la^{[2]},\mu^{[2]}}\,|\mu,\Bs\rp,
\end{align*}
we have, by Proposition 2.13, 
\begin{align*}
& \ol{\CG^{\pm}(\la^{[1]},\Bs^{[1]})\otimes 
   \CG^{\pm}(\la^{[2]},\Bs^{[2]})}\\
&\qquad = \sum_{\mu\in \Pi^l} \ol{\D}^{\,\pm}_{\la^{[1]},\mu^{[1]}}
    \ol{\D}^{\,\pm}_{\la^{[2]},\mu^{[2]}}\, \ol{|\mu,\Bs\rp}\\
&\qquad = \sum_{\mu\in \Pi^l} \ol{\D}^{\,\pm}_{\la^{[1]},\mu^{[1]}}
    \ol{\D}^{\,\pm}_{\la^{[2]},\mu^{[2]}}\,
	\Big\{ \ol{|\mu^{[1]},\Bs^{[1]}\rp }\otimes 
    \ol{|\mu^{[2]},\Bs^{[2]}\rp}
	+ \sum_{\nu\in \Pi^l \atop \Ba(\mu) >  \Ba(\nu) }
   \a_{\mu,\nu}\,|\nu, \Bs\rp\Big\}\\
&\qquad = \ol{\CG^{\pm}(\la^{[1]},\Bs^{[1]})}\otimes 
   \ol{\CG^{\pm}(\la^{[2]},\Bs^{[2]})}\\
	&\hspace{6em}+\sum_{\mu\in \Pi^l} 
  \ol{\D}^{\,\pm}_{\la^{[1]},\mu^{[1]}}
  \ol{\D}^{\,\pm}_{\la^{[2]},\mu^{[2]}}\,
	\Big\{\sum_{\nu\in \Pi^l \atop \Ba(\mu) > \Ba(\nu) }
\a_{\mu,\nu}\,|\nu^{[1]}, \Bs^{[1]}\rp \otimes |\nu^{[2]},\Bs^{[2]}\rp \Big\}.
\end{align*}
By the property of the canonical bases, we have 
$\ol{\CG^{\pm}(\la^{[i]},\Bs^{[i]})} = 
   \CG^{\pm}(\la^{[i]},\Bs^{[i]})$ for $i=1,2$. 
Note that, $|\la^{[i]}|=|\mu^{[i]}|$ 
if $\D^{\pm}_{\la^{[i]},\mu^{[i]}}\not=0$ for $i=1,2$. 
Moreover, a vector $|\nu^{[i]},\Bs^{[i]}\rp$ 
can be written as a linear combination of the canonical bases 
$\CG^{\pm}(\ka^{[i]},\Bs^{[i]})$ such that 
$|\ka^{[i]}|=|\nu^{[i]}|$ for $i=1,2$. 
Hence we have 
\begin{equation*}
\tag{2.15.1}
\begin{split} 
& \ol{\CG^{\pm}(\la^{[1]},\Bs^{[1]})\otimes 
 \CG^{\pm}(\la^{[2]},\Bs^{[2]})} \\
& =\CG^{\pm}(\la^{[1]},\Bs^{[1]}) \otimes
 \CG^{\pm}(\la^{[2]},\Bs^{[2]})
		+\!\!\sum_{\mu \in \Pi^l \atop 
                 \Ba(\la) > \Ba(\mu) }\!\! b'^{\pm}_{\la,\mu} \CG^{\pm}
(\mu^{[1]},\Bs^{[1]}) \otimes \CG^{\pm}(\mu^{[2]},\Bs^{[2]})
\end{split}
\end{equation*}

with ${b'}^{\pm}_{\la,\mu}\in \QQ[q,q\iv]$. 
Thus one can write as 
\begin{equation*}
\ol{\CG^{\pm}(\la^{[1]},\Bs^{[1]})\otimes \CG^{\pm}(\la^{[2]},\Bs^{[2]})}
	=\sum_{\mu\in \Pi^l \atop |\mu|=|\la|}R^{\pm}_{\la,\mu}
  \,\CG^{\pm}(\mu^{[1]},\Bs^{[1]})\otimes
  \CG^{\pm}(\mu^{[2]},\Bs^{[2]}),
\end{equation*}
with $R^{\pm}_{\la,\mu}\in \QQ[q,q\iv]$, where  
the matrix $\big(R^{\pm}_{\la,\mu}\big)_{|\la|=|\mu|}$ is 
 unitriangular with respect to the order compatible with
$\Ba(\la) > \Ba(\mu)$ by (2.15.1). 
Thus, by a standard argument for constructing the canonical bases, we have 
\begin{equation*}
\begin{split}
\CG^{\pm}(\la,\Bs) = \CG^{\pm}(&\la^{[1]},\Bs^{[1]})
  \otimes\CG^{\pm}(\la^{[2]},\Bs^{[2]})  \\
	&+ \sum_{\mu \in \Pi^l \atop \Ba(\la) >  \Ba(\mu)}
 b^{\pm}_{\la,\mu}\,\CG^{\pm}(\mu^{[1]},\Bs^{[1]})\otimes\CG^{\pm}
(\mu^{[2]},\Bs^{[2]}),
\end{split}
\end{equation*}
with $b^{\pm}_{\la,\mu}\in \QQ[q^{\pm 1}]$. This proves 
Theorem 2.15.
\end{proof}
\para{2.16.}  We now prove Theorem 2.9 in the case where $g = 2$,
assuming that Proposition 2.13 holds. 
Assume that $|\la,\Bs\rp$ is $M$-dominant for $M > 2n$.
Since $\CG(\la^{[i]}, \Bs^{[i]}) = 
       \sum_{\mu^{[i]} \in \Pi^{l_i}}\vD^{\pm}_{\la^{[i]},\mu^{[i]}}
                 |\mu^{[i]},\Bs^{[i]}\rp$,
it follows from Theorem 2.15 that 
\begin{equation*}
\begin{split}
\CG^{\pm}(\la,\Bs) = \sum_{\mu \in \Pi^l \atop \a_{\Bp}(\la) = 
\a_{\Bp}(\mu)}\vD^{\pm}_{\la^{[1]},\mu^{[1]}}
              \vD^{\pm}_{\la^{[2]},\mu^{[2]}} |\mu,\Bs\rp 
   + \sum_{\mu \in \Pi^l \atop \Ba(\la) > \Ba(\mu)}
            \wt b^{\pm}_{\la,\mu}|\mu,\Bs\rp.
\end{split}
\end{equation*}
Since $\a_{\Bp}(\la) = \a_{\Bp}(\mu)$ is equivalent to 
$\Ba(\la) = \Ba(\mu)$, this implies 
that for any $\mu \in \Pi^l$ such that 
$\a_{\Bp}(\la) = \a_{\Bp}(\mu)$, 
\begin{equation*}
\tag{2.16.1}
\vD^{\pm}_{\la,\mu}(q) = \vD^{\pm}_{\la^{[1]},\mu^{[1]}}(q)
                      \vD^{\pm}_{\la^{[2]},\mu^{[2]}}(q).
\end{equation*}
This proves Theorem 2.9 modulo Proposition 2.13.
\remark{2.17.} By [U, Th. 3.26], $\vD^{\pm}_{\la\mu}(q)$
can be interpreted by parabolic Kazhdan-Lusztig polynomials
of an affine Weyl group. So, Theorem 2.9 gives a product 
formula for parabolic Kazhdan-Lusztig polynomials.  
It would be interesting to give a geometric interpretation 
of this formula. 
\section{ Tensor product of the Fock spaces}
\para{3.1.}
In this section, we prove Proposition 2.13 after some 
preliminaries. The proof will be given in 3.11.
For a given $u_{\Bk} = |\la, \Bs\rp$, we associate semi-infinite sequences
$\Bk^{(b)} = (k_1^{(b)}, k_2^{(b)}, \dots) \in \BP^{++}(s_b)$ 
for $b = 1, \dots, l$ as in 2.3.
For $b\in \{1,\cdots ,l\}$ and $k\in \ZZ$, we put 
\begin{equation*}
\tag{3.1.1}
u_k^{(b)}=u_{a+n(b-1)-nlm}, 
\end{equation*} 
where $a\in \{1,\cdots,n\}$ and $m\in \ZZ$ are uniquely 
determined by $k= a - nm$. 
For a positive integer $r$, put $\Bk_r = (k_1, k_2, \dots, k_r)$.
Then $u_{\Bk}$ can be written as 
$u_{\Bk} = u_{\Bk_r}\we u_{k_r+1}\we u_{k_r+2} \we \cdots$, 
where $u_{\Bk_r} = u_{k_1}\we\cdots \we u_{k_r} \in \vL^r$, which 
is called a finite wedge of length $r$.
We also define a finite wedge 
$u_{\Bk_r^+} \in \vL^r$ for a sufficiently large $r$ by 
\begin{equation*}
\tag{3.1.2}
u_{\Bk_r^+} = u^{(1)}_{\Bk_{r_1}^{(1)}}\we u^{(2)}_{\Bk_{r_2}^{(2)}}
                   \we\cdots \we u^{(l)}_{\Bk_{r_l}^{(l)}}
\end{equation*}
with $u^{(i)}_{\Bk_{r_i}^{(i)}} = 
       u^{(i)}_{k_1^{(i)}} \we \cdots \we u^{(i)}_{k_{r_i}^{(i)}}$ 
for $i = 1, \dots, l$, where each $r_i$ is sufficiently large  
and $r=r_1+\cdots +r_l$.  
Then in view of 2.3, we see that $u_{\Bk^+_r}$ is obtained 
from $u_{\Bk_r}$ by permuting the sequence 
$\Bk_r$.  Moreover, 
$u_{k_1^{(b)}}\we\cdots \we u_{k_{r_b}^{(b)}}$ 
is the first $r_b$-part of the wedge $u_{\Bk^{(b)}}$ corresponding
to $|\la^{(b)}, s_b\rp$  (under the correspondence for the case $l =1$ 
in 2.3) for each $b = 1, \dots, l$.
\par
Note that $u_{\Bk_r^+}$ is not necessarily ordered in general, 
and it is written as a linear combination of ordered wedges. 
But the situation becomes drastically simple under 
the assumption on $M$-dominance.
We have following two lemmas due to Uglov.
\begin{lem}[{[U, Lemma 5.18]}] 
Let $b_1, b_2 \in \{ 1, \dots, l\}$ and $a_1, a_2 \in \{ 1, \dots, n\}$, 
and assume that $b_1 < b_2, a_1 \ge a_2$.  For any $m \in \ZZ$, 
$t \in \ZZ_{\ge 0}$, put 
$X = u_{a_1-nm}^{(b_1)}\we u_{a_1-nm-1}^{(b_1)}\we \cdots\we 
  u_{a_1-nm-t}^{(b_1)}$.  Then there exists $c \in \ZZ$ 
such that the following relation holds.
\begin{equation*}
X \we u_{a_2-nm}^{(b_2)} = q^c u_{a_2-nm}^{(b_2)}\we X.
\end{equation*}
\end{lem}
\begin{lem}[{[U, Lemma 5.19]}] 
Take $\la \in \Pi^l$ and $\Bs \in \ZZ^l(s)$, and assume that
 $|\la,\Bs \rp $ is  $0$-dominant.  Then under the notation of 2.3, 
we have
\begin{align*}
|\la,\Bs\rp =u_{\Bk}&=(u_{k_1}\we u_{k_2}\we \cdots \we u_{k_r}) 
         \we u_{k_{r+1}}\we u_{k_{r+2}} \we \cdots \\
	&=q^{-c_r(\Bk)} \,u_{\Bk_r^+}\we u_{k_{r+1}}
            \we u_{k_{r+2}} \cdots 
\end{align*}
where $c_r(\Bk)=\sharp \{1\leqq i<j \leqq r \,|\, b_i>b_j,\,a_i=a_j \}$. 
\end{lem}
\para{3.4.}
Returning to the setting in 2.11, we describe the map $\F$ in terms 
of the wedges. Take $u_{\Bk} = |\la,\Bs\rp$, and assume that 
$|\la, \Bs\rp$ is 0-dominant.  
Then by Lemma 3.3, we can write as 
\begin{equation*}
\tag{3.4.1}
u_{\Bk} = q^{-c_r(\Bk)}\bigl(u^{(1)}_{\Bk_{r_1}^{(1)}}
             \we\cdots\we u^{(t)}_{\Bk_{r_{t}}^{(t)}}\bigr)\we
      \bigl(u^{(t+1)}_{\Bk_{r_{t+1}}^{(t+1)}}\we\cdots\we
             u^{(l)}_{\Bk_{r_l}^{(l)}}\bigr)\we u_{r+1}\we \cdots, 
\end{equation*}
where $r_1, \dots, r_l$ are sufficiently large.
Let $u'_{\Bk'}$ (resp. $u''_{\Bk''}$) be the ordered wedge 
in $\vL^{s'+\frac{\infty}{2}}$ (resp. in $\vL^{s''+\frac{\infty}{2}}$)
corresponding to $|\la^{[1]}, \Bs^{[1]}\rp$ 
(resp. $|\la^{[2]}, \Bs^{[2]}\rp \ $).
Since $|\la, \Bs\rp$ is 0-dominant, $|\la^{[i]}, \Bs^{[i]}\rp$ is
0-dominant for $i = 1, 2$.
Hence again by Lemma 3.3, we have
\begin{equation*}
\tag{3.4.2}
\begin{aligned}
u'_{\Bk'} &= q^{-c_{r'}(\Bk')}
    \bigl({u'}^{(1)}_{{\Bk}_{r_1}^{(1)}}\we\cdots\we
          {u'}^{(t)}_{{\Bk}_{r_{t}}^{(t)}}\bigr)
                         \we u'_{k_{r'+1}}\we\cdots \\
u''_{\Bk''} &=
 q^{-c_{r''}(\Bk'')}\bigl({u''}^{(1)}_{{\Bk}_{r_{t+1}}^{(t+1)}}\we
     \cdots\we {u''}^{(l-t)}_{{\Bk}_{r_l}^{(l)}}\bigl)
                 \we u''_{k_{r''+1}}\we\cdots,
\end{aligned}
\end{equation*}
where $r' = r_1 + \cdots + r_{t}$ and 
$r'' = r_{t+1} + \cdots r_l$.
Thus, in the case where $|\la,\Bs\rp$ is 0-dominant, the map
$\F : u'_{\Bk'}\otimes u''_{\Bk''} \mapsto u_{\Bk}$ is 
obtained by attaching
\begin{align*}
{u'}^{(1)}_{{\Bk}_{r_1}^{(1)}}\we\cdots\we
          {u'}^{(t)}_{{\Bk}_{r_{t}}^{(t)}} &\mapsto
    u^{(1)}_{\Bk_{r_1}^{(1)}}
             \we\cdots\we u^{(t)}_{\Bk_{r_{t}}^{(t)}}  \\
{u''}^{(1)}_{{\Bk}_{r_{t+1}}^{(t+1)}}\we
     \cdots\we {u''}^{(l-t)}_{{\Bk}_{r_l}^{(l)}}  &\mapsto
   u^{(t+1)}_{\Bk_{r_{t+1}}^{(t+1)}}\we\cdots\we
             u^{(l)}_{\Bk_{r_l}^{(l)}}
\end{align*}
and by adjusting the power of $q$.
\para{3.5.}
We are mainly concerned with the expression as in the right hand 
side of (3.4.1), instead of treating $u_k$ directly. So we will 
modify the ordering rule so as to fit the expression 
by $u_k^{(b)}$.  Recall that $u_k^{(b)} = u_{a + n(b-1) -nlm}$ 
in (3.1.1).  We define a total order on the set
$\{ u^{(b)}_k \mid b \in \{ 1, \dots, l\}, k \in \ZZ\}$ 
by inheriting the total order on the set 
$\{ u_k \mid k \in \ZZ\} \simeq \ZZ$.  
The following property is easily verified. 
\par\medskip\noindent
(3.5.1) \ Assume that $u_{k_i}^{(b_i)} = u_{a_i + n(b_i-1) - nlm_i}$
for $i = 1,2$.  Then 
$u_{k_2}^{(b_2)} < u_{k_1}^{(b_1)}$ if
and only if one of the following three cases occurs;
\begin{enumerate}
\item $m_1 <  m_2$, 
\item $m_1 = m_2$ and $b_1 > b_2$, 
\item $m_1 = m_2$, $b_1 = b_2$ and $a_1 > a_2$.
\end{enumerate}
Under this setting, the ordering  rule in [U, Prop. 3.16] 
can be rewritten as follows. 
\begin{prop} 
\begin{enumerate}
\item
Suppose that $u_{k_2}^{(b_2)}\le u_{k_1}^{(b_1)}$
for $b_1,b_2\in \{1,\cdots ,l\}$, $k_1,k_2\in \ZZ$. 
Let $\g$ be the residue of $k_1 - k_2$ modulo $n$. 
Then we have the following formulas.
\par\medskip\noindent
{\rm(R1)} \ the case where $\g = 0$ and $b_1 = b_2$,
\begin{equation*}
u_{k_2}^{(b_2)} \we u_{k_1}^{(b_1)}
    =-u_{k_1}^{(b_1)} \we u_{k_2}^{(b_2)},
\end{equation*}
{\rm(R2)} \ the case where $\g \ne 0$ and $b_1 = b_2$,
\begin{align*} 
u_{k_2}^{(b_2)} \we u_{k_1}^{(b_1)}=
       &-q^{-1}u_{k_1}^{(b_1)} \we u_{k_2}^{(b_2)}  \\ 
       &+(q^{-2}-1) \sum_{m\ge 0}q^{-2m} u_{k_1-\g -nm}^{(b_1)} 
           \we u_{k_2+\g +nm}^{(b_2)}  \\
       &-(q^{-2}-1) \sum_{m\ge 1}q^{-2m+1} u_{k_1-nm}^{(b_1)} 
          \we u_{k_2+nm}^{(b_2)}, 
\end{align*}
{\rm(R3)} the case where $\g = 0$ and $b_1 \ne b_2$,
\begin{align*}
u_{k_2}^{(b_2)} \we u_{k_1}^{(b_1)}= 
    &q u_{k_1}^{(b_1)} \we  u_{k_2}^{(b_2)} \\  
		&+(q^{2}-1) \sum_{m\ge \ve }q^{2m} u_{k_1 -nm}^{(b_2)} 
                   \we u_{k_2 +nm}^{(b_1)} \\
		&+(q^{2}-1) \sum_{m\ge 1}q^{-2m+1} u_{k_1-nm}^{(b_1)} 
                    \we u_{k_2+nm}^{(b_2)}, 
\end{align*}
{\rm(R4)} the case where $\g \ne 0$ and $b_1 \ne b_2$,
\begin{align*}
u_{k_2}^{(b_2)} \we u_{k_1}^{(b_1)}=  &u_{k_1}^{(b_1)} \we
 u_{k_2}^{(b_2)} \\  
			&+(q-q^{-1})\sum_{m\geq 0}
 \frac{(q^{2m+1}+q^{-2m-1})}
          {(q+q^{-1})}u_{k_1-\g -nm}^{(b_1)} \we u_{k_2+\g +nm}^{(b_2)} \\
	&+(q-q^{-1})\sum_{m\ge \ve}	
                      \frac{(q^{2m+1}+q^{-2m-1})}{(q+q^{-1})}
		u_{k_1 -nm}^{(b_2)} \we u_{k_2 +nm}^{(b_1)}  \\
	&+(q-q^{-1})\sum_{m\ge \ve}\frac{(q^{2m} -q^{-2m})}{(q+q^{-1})}
		u_{k_1-\g -nm}^{(b_2)}\we u_{k_2+\g +nm}^{(b_1)}  \\
	&+(q-q^{-1})\sum_{m\ge 1}\frac{(q^{2m} -q^{-2m})}{(q+q^{-1})} 
                u_{k_1-nm}^{(b_1)} \we u_{k_2+nm}^{(b_2)},
\end{align*}
where in the formula (R3) and (R4), 
\begin{equation*}
\ve = \begin{cases}
          1 \quad\text{ if } b_1 < b_2, \\
          0 \quad\text{ if } b_1 > b_2.
       \end{cases}
\end{equation*}
The sums are taken over all $m$ such that the wedges in the sum 
remain ordered. 
\item
For a wedge 
$u_{k_1}^{(b_1)}\we u_{k_2}^{(b_2)} \we \cdots \we u_{k_r}^{(b_r)}$, 
above relations hold in every pair of adjacent factors.
\end{enumerate}
\end{prop}
\remark{3.7.}
The ordering rule in the proposition does not depend on the
choice of $l$.  It depends only on $k_1, k_2$ and whether 
$b_1 = b_2$ or not.  This implies the following.  
Assume that $\Bp = (t, l-t), s = s' + s''$ 
and 
${u'}_k^{(b)} \in \vL^{s' + \hinf}$,  
${u''}_{k}^{(b)} \in \vL^{s''+\hinf}$ as before.  Then if   
$b_1, b_2 \in \{ 1, 2, \dots, t \}$, the ordering rule
for $u_{k_2}^{(b_2)} \le u_{k_1}^{(b_1)}$ is the same as 
the rule for ${u'}_{k_2}^{(b_2)} \le {u'}_{k_1}^{(b_1)}$. 
Similarly, if $b_1, b_2 \in \{ t+1, \dots, l \}$ the 
rule for $u_{k_2}^{(b_2)} \le u_{k_1}^{(b_1)}$ is the same 
as the rule for ${u''}_{k_2}^{(b_2-t)} \le {u''}_{k_1}^{(b_1-t)}$.
\par\medskip
We show the following three lemmas.
\begin{lem} 
Let $M$ be an integer such that $M > 2n$. 
Assume that $u_{\Bk}=|\la,\Bs\rp$ is $M$-dominant. 
For $b\in \{1,\cdots ,l\}$ and $i\in \ZZ$, put 
$k_i^{(b)}=s_b-i+1+\la_i^{(b)}= 
 a_i^{(b)}-nm_i^{(b)}\,\,(a_i^{(b)}\in \{1,\cdots,n\},\,m_i^{(b)}\in \ZZ)$. 
Fix $b_1, b_2 \in \{1,\cdots, l\}$ such that $b_1 <b_2$. 
For $k_i^{(b_2)}$, let $\s(i)$ be the smallest $j$ such that 
$u_{k_i^{(b_2)}}^{(b_2)} > u_{k_j^{(b_1)}}^{(b_1)}$. 
Then we have
\begin{enumerate}
\item $\la_{\s(i)}^{(b_1)} = 0$.
\item $k_{\s(i)}^{(b_1)} = n - nm_i^{(b_2)}$.
\end{enumerate}
\end{lem}
\begin{proof}
Let $\ell(\mu)$ be the number of non-zero parts of a partition $\mu$.
We put $p = \ell(\la^{(b_1)})+1$. 
In order to show (i), it is enough to see 
\begin{equation*}
\tag{3.8.1}
a_p^{(b_1)} + n(b_1-1) - nlm_p^{(b_1)} > 
         a_1^{(b_2)} + n(b_2-1) - nlm_1^{(b_2)}
\end{equation*}
In fact, by (3.8.1), we have 
$u^{(b_1)}_{k^{(b_1)}_p} > u^{(b_1)}_{k^{(b_1)}_{\s(i)}}$  
since $u_{k_1^{(b_2)}}^{(b_2)} \ge u_{k_i^{(b_2)}}^{(b_2)} 
         > u_{k_{\s(i)}^{(b_1)}}^{(b_1)}$. 
This implies  that $p < \s(i)$. Since $\la^{(b_1)}_p = 0$, 
we have $\la_{\s(i)}^{(b_1)}=0$. 
\par
We show (3.8.1). 
If we put 
\begin{equation*}
X = (a_p^{(b_1)} + n(b_1-1) - nlm_p^{(b_1)}) - 
         (a_1^{(b_2)} + n(b_2-1) -  nlm_1^{(b_2)}),
\end{equation*}
we have 
\begin{align*}
\tag{3.8.2}
X = l\{(a_p^{(b_1)} - &nm_p^{(b_1)}) - (a_1^{(b_2)} - nm_1^{(b_2)})\} \\ 
    &-(l-1)(a_p^{(b_1)}-a_1^{(b_2)})-n(b_2-b_1).  
\end{align*}
Since
\begin{align*}
 k_1^{(b_2)} &= s_{b_2}+\la^{(b_2)}_1 = a_1^{(b_2)}-nm_1^{(b_2)}, \\
 k_p^{(b_1)} &= s_{b_1}-\ell(\la^{(b_1)}) = a_p^{(b_1)}-nm_p^{(b_1)},
\end{align*}
by replacing $a_1^{(b_2)}-nm_1^{(b_2)}$ by $s_{b_2}+\la^{(b_2)}_1$, and 
similarly for $a_p^{(b_1)} - nm_p^{(b_1)}$ in (3.8.2), we see that 
\begin{align*}
X &=l\big\{\bigl(s_{b_1}-\ell(\la^{(b_1)})\bigr) - 
  \big(s_{b_2}+\la_1^{(b_2)}\big)\big\} -
      (l-1)(a_p^{(b_1)}-a_1^{(b_2)})-n(b_2-b_1) \\
&=l(s_{b_1}-s_{b_2})-l\big(\ell(\la^{(b_1)}) - 
    \la_1^{(b_2)})-(l-1)(a_p^{(b_1)}-a_1^{(b_2)})-n(b_2-b_1) \\
&\ge l(s_{b_1}-s_{b_2})-l|\la|-ln-nl  \\
&= l\big\{(s_{b_1}-s_{b_2})-(|\la|+2n)\big\}.
\end{align*}
Since $M>2n$, we have $s_{b_1}-s_{b_2}\ge |\la|+M>|\la|+2n$, and 
so $X >0$. This proves (3.8.1) and (i) follows. 
\par
Next we show (ii). 
By definition, $\s(i)$ is the smallest integer $j$ such that
\begin{equation*}
a_i^{(b_2)} + n(b_2 - 1) - nlm_i^{(b_2)} 
    > a_j^{(b_1)} + n(b_1 - 1)- nlm_j^{(b_1)}.
\end{equation*}
If for $a \in \{ 1, \dots, n\}$ amd $m \in \ZZ$, 
\begin{equation*}
\tag{3.8.3}
a_i^{(b_2)} + n(b_2 - 1) - nlm_i^{(b_2)} > a + n(b_1 - 1) - nlm,
\end{equation*}
then we have $k_p^{(b_1)} > a - nm$ by (3.8.1) and (3.5.1).
Note that 
$k_j^{(b_1)}=a_j^{(b_1)}-nm_j^{(b_1)} = s_{b_1}-j+1+\la_j^{(b_1)}$ 
and $\la_j^{(b_1)}=0$ for any $j \ge p$.  It follows that  
$k_{p + j}^{(b_1)}=k_{p}^{(b_1)} - j$ for any $j \ge 1$. 
Hence there exists an integer $j_0 \ge 1$ such that 
$k_{p + j_0}^{(b_1)} = a - nm$. 
This means that $k_{\s(i)}^{(b_1)}$ is the largest integer 
$a - nm$ for $a \in \{ 1, \dots, m\}, m \in \ZZ$ satisfying 
the inequality (3.8.3).  
Clearly, $a - nm = n - nm_i^{(b_2)}$ is the largest, and 
we obtain (ii). 
\end{proof}
\begin{lem}  
For $b_1,b_2 \in \{1,\cdots,l\}$ such that 
$b_1 < b_2$,  and $k_1,k_2 \in\ZZ$ such that 
$u_{k_2}^{(b_2)} < u_{k_1}^{(b_1)}$, 
we have 
\begin{equation*}
\tag{3.9.1}
u_{k_2}^{(b_2)}\we u_{k_1}^{(b_1)} = 
     \a(k_1, k_2)\, u_{k_1}^{(b_1)}\we u_{k_2}^{(b_2)}
 +\sum_{(k'_1,k'_2)\in \ZZ^2 \atop k_1 > k'_1 , k'_2 > k_2} 
\a(k'_1,k'_2)\, u_{k'_1}^{(b_1)}\we u_{k'_2}^{(b_2)}
\end{equation*}
with  $\a(k_1,k_2), \a(k'_1,k'_2)\in \QQ[q, q\iv]$.
\end{lem}
\begin{proof}
Put $k_i=a_i-nm_i$ for $i = 1,2$, 
where $a_i\in \{1,\cdots ,n\},\,m_i\in \ZZ$.
First assume that $a_1=a_2$. By the ordering rule (R3) 
in Proposition 3.6, we have 
\begin{equation*}
\tag{3.9.2}
\begin{aligned}
u_{k_2}^{(b_2)}\we u_{k_1}^{(b_1)}=&q u_{k_1}^{(b_1)} 
    \we u_{k_2}^{(b_2)}  \label{qqq}\\* 
		&+(q^{2}-1) \sum_{m\geq 1 }q^{2m} 
         u_{k_1 -nm}^{(b_2)} \we u_{k_2 +nm}^{(b_1)}  \\ 
	&+(q^{2}-1) \sum_{m\geq 1}q^{-2m+1} u_{k_1-nm}^{(b_1)} 
          \we u_{k_2+nm}^{(b_2)},  
\end{aligned}
\end{equation*}
and the only ordered wedges appear in the sums.
Note that $a_1 = a_2$, $b_1 < b_2$ and $m\ge 1$. 
Then in the second sum, 
the condition $u_{k_1-nm}^{(b_1)} > u_{k_2+nm}^{(b_2)}$ implies that 
$k_1 > k_1-nm > k_2+nm > k_2$ by (3.5.1).   
It follows that the terms in the second sum are all of the form 
$u_{k_1'}^{(b_1)}\we u_{k_2'}^{(b_2)}$ as in (3.9.1). 
On the other hand, in the first sum, the condition 
$u_{k_1 -nm}^{(b_2)} > u_{k_2 +nm}^{(b_1)} $ implies that 
$k_1>k_1-nm \geq k_2+nm>k_2$ by (3.5.1). 
Hence if $k_1-k_2< 2n$, the terms 
$u_{k_1 -nm}^{(b_2)} \we u_{k_2 +nm}^{(b_1)} $ do not appear 
in the sum.  So assume that $k_1 - k_2 \ge 2n$. 
We apply the ordering rule (R3) to 
$u_{k_2 +nm}^{(b_1)} \we u_{k_1 -nm}^{(b_2)}$, and we obtain
\begin{equation*}
\tag{3.9.3}
u_{k_1 -nm}^{(b_2)}\we u_{k_2+nm}^{(b_1)}
   = q\iv u_{k_2+nm}^{(b_1)}\we u_{k_1-nm}^{(b_2)} 
      + X_1 + X_2, 
\end{equation*}
where 
$X_1$ (resp. $X_2$) is a linear combination of the wedges 
$u_{k_1-nm'}^{(b_2)}\we u_{k_2+nm'}^{(b_1)}$ 
(resp. $u_{k_1-nm'}^{(b_1)} \we u_{k_2+nm'}^{(b_2)}$)
with $m' > m$.
Note that $k_1 > k_2 + nm$ and $k_1-nm > k_2$, 
and so $u_{k_2+nm}^{(b_1)}\we u_{k_1-nm}^{(b_2)}$
is of the form $u_{k_1'}^{(b_1)}\we u_{k_2'}^{(b_2)}$ 
in (3.9.1).
We can apply the same procedure as above for replacing 
$u_{k_1 -nm'}^{(b_2)} \we u_{k_2 +nm'}^{(b_1)}$  in $X_1$ by the 
terms $u_{k_1'}^{(b_1)}\we u_{k_2'}^{(b_2)}$ and other terms.
Repeating this procedure, finally we obtain 
the expression as in (3.9.1). 
Note that since 
$2n \le k_1' - k_2' < k_1 - k_2$ for $k_1' = k_1 -nm$ and 
$k_2' = k_2 + nm$, this procedure will end up after finitely 
many steps. 
\par
Next consider the case where $a_1 \ne a_2$.  In this case 
we apply the ordering rule (R4).
Then one can write as 
\begin{equation*}
u_{k_2}^{(b_2)}\we u_{k_1}^{(b_1)} = 
     u_{k_1}^{(b_1)}\we u_{k_2}^{(b_2)} + X_1 + X_2 + X_3 + X_4,
\end{equation*}
where $X_1, \dots, X_4$ are the corresponding sums in (R4) with 
$\ve = 1$.  For $X_1, X_2, X_4$, similar arguments as above can be
applied.  So we have only to consider the sum $X_3$ which contains
the terms of the form $u_{k_1-\g-nm}^{(b_2)}\we u_{k_2+\g+nm}^{(b_1)}$.
In this case, the condition 
$u_{k_1-\g-nm}^{(b_2)} > u_{k_2+\g+nm}^{(b_1)}$
implies that $(k_1 - nm) - (k_2 + nm) \ge a_1 - a_2 > -n$, and so
$k_1 - k_2 > n(2m-1)$.  Hence, if $k_1 - k_2 \le n$, 
the terms $u_{k_1-\g-nm}^{(b_2)}\we u_{k_2+\g+nm}^{(b_1)}$ do not
appear in the sum.  If $k_1 - k_2 > n $, we can apply a similar argument
as before by using the ordering rule (R4).
Thus we obtain the expression in (3.9.1) in this case also.  
The lemma is proved.
\end{proof}
\begin{lem}  
Assume that $u_{\Bk}=|\la,\Bs\rp$ is $M$-dominant for
$M > 2n$. 
For $b\in \{1,\cdots ,l\}$ and $i\in \ZZ$, put 
$k_i^{(b)}=s_b-i+1+\la_i^{(b)}$. 
Then for $d\in \{t+1,t+2, \cdots ,l\}$ and $i\in \ZZ$, we have 
\begin{align*}u_{k_i^{(d)}}^{(d)}\we &\big( u_{k_{r_t}^{(t)}}^{(t)} 
  \we u_{k_{r_t-1}^{(t)}}^{(t)}\we \cdots \we u_{k_{1}^{(t)}}^{(t)} \we 
	u_{k_{r_{t-1}}^{(t-1)}}^{(t-1)}\we \cdots \we 
   u_{k_{2}^{(1)}}^{(1)} \we u_{k_{1}^{(1)}}^{(1)} \big)\\
&= \a\big( u_{k_{r_t}^{(t)}}^{(t)} \we 
 u_{k_{r_t-1}^{(t)}}^{(t)}\we \cdots \we u_{k_{1}^{(t)}}^{(t)} \we 
	u_{k_{r_{t-1}}^{(t-1)}}^{(t-1)}\we \cdots \we 
   u_{k_{2}^{(1)}}^{(1)} \we u_{k_{1}^{(1)}}^{(1)} \big)
     \we u_{k_i^{(d)}}^{(d)} + Y_1, \\
\end{align*}
where $\a \in \QQ[q,q\iv]$ and $Y_1$ is a $\QQ[q,q\iv]$-linear combination 
of the wedges of the form
\begin{equation*}
\bigl( u_{\wt{k}_{r_t}^{(t)}}^{(t)} \we u_{\wt{k}_{r_t-1}^{(t)}}^{(t)}
\we \cdots \we u_{\wt{k}_{1}^{(1)}}^{(1)}\bigr)\we u_{\wt k_i^{(d)}}
\quad\text{ for }\quad 
  (\wt k_{r_t}^{(t)}, \dots, \wt k_1^{(1)}; \wt k_i^{(d)}) \in
                         \ZZ^{r'}\times \ZZ
\end{equation*}
under the condition
\begin{equation*}
 {k}_{r_t}^{(t)} + \cdots + {k}_1^{(1)} > 
       \wt{k}_{r_t}^{(t)}+ \cdots + \wt{k}_1^{(1)}, \quad
{k}_i^{(d)} < \wt{k}_i^{(d)}.
\end{equation*}
\end{lem}
\begin{proof}
Put $k_i^{(d)}=a_i-nm_i$. 
Let $\s(i)$ be the smallest $j$ such that 
$u_{k_i^{(d)}}^{(d)}>u_{k_j^{(t)}}^{(t)}$. 
Since $d > t$, by applying Lemma 3.8, we have 
$k_{\s(i)}^{(t)} = n- nm_i$ and $\la^{(t)}_{\s(i)} = 0$. 
Thus, we have 
\begin{align*}
u_{k_i^{(d)}}^{(d)}\we \bigl(u_{k_{r_t}^{(t)}}^{(t)}&\we \cdots 
  \we u_{k_{\s(i)+1}^{(t)}}^{(t)}\we u_{k_{\s(i)}^{(t)}}^{(t)}\bigr)\\
	&=u_{a_i-nm_i}^{(d)}\we \bigl(u_{n-nm_i-(r_t-\s(i))}^{(t)} 
    \we \cdots \we u_{n-nm_i-1}^{(t)} \we u_{n-nm_i}^{(t)}\bigr).
\end{align*}
Using the formula obtained by applying the bar-involution 
on the formula in Lemma 3.2, we have 
\begin{align*}
u_{k_i^{(d)}}^{(d)}\we \bigl(u_{k_{r_t}^{(t)}}^{(t)}&\we \cdots \we 
u_{k_{\xi(i)+1}^{(t)}}^{(t)}\we u_{k_{\s(i)}^{(t)}}^{(t)}\bigr) \\
	&=\b \bigl(u_{k_{r_t}^{(t)}}^{(t)}\we \cdots \we 
u_{k_{\s(i)+1}^{(t)}}^{(t)}\we u_{k_{\s(i)}^{(t)}}^{(t)}\bigr)\we 
u_{k_i^{(d)}}^{(d)} 
\end{align*}
with $\b \in \QQ[q,q\iv]$. 
Since $u_{k_i^{(d)}}^{(d)}<u_{k_{\s(i)-1}^{(t)}}^{(t)}< \cdots 
  < u_{k_1^{(t)}}^{(t)}$, using Lemma 3.9 repeatedly, we have 
\begin{align*}
u_{k_i^{(d)}}^{(d)}\we \bigl(u_{k_{r_t}^{(t)}}^{(t)}&\we \cdots 
  \we u_{k_{\s(i)+1}^{(t)}}^{(t)}\we u_{k_{\s(i)}^{(t)}}^{(t)} 
\we u_{k_{\s(i)-1}^{(t)}}^{(t)}\we \cdots \we u_{k_1^{(t)}}^{(t)}\bigr) \\
&= \b\bigl(u_{k_{r_t}^{(t)}}^{(t)}\we \cdots \we
 u_{k_{\s(i)+1}^{(t)}}^{(t)}
\we u_{k_{\s(i)}^{(t)}}^{(t)}\bigr) \we u_{k_i^{(d)}}^{(d)} \we
	\bigl(u_{k_{\s(i)-1}^{(t)}}^{(t)}\we \cdots \we 
   u_{k_1^{(t)}}^{(t)}\bigr) \\
&=\wt{\b}\bigl( u_{k_{r_t}^{(t)}}^{(t)} \we \cdots \we 
u_{k_1^{(t)}}^{(t)}\bigr) \we u_{k_i^{(d)}}^{(d)}  + Y'_1, 
\end{align*}
where $\wt\b \in \QQ[q,q\iv]$ and $Y'_1$ is a $\QQ[q,q\iv]$-linear 
combination of the wedges
of the form
\begin{equation*} 
\bigl( u_{\wt{k}_{r_t}^{(t)}}^{(t)}\we \cdots \we 
         u_{\wt{k}_1^{(t)}}^{(t)})\we u_{\wt{k_i}^{(d)}}^{(d)}
\quad\text{ for }\quad
(\wt{k}_{r_t}^{(t)}, \dots, \wt{k}_1^{(t)}; \wt{k}_i^{(d)}) 
      \in \ZZ^{r_t}\times \ZZ,
\end{equation*}
under the condition 
\begin{equation*}
k_{r_t}^{(t)}+ \dots + k_1^{(t)} > 
    \wt k_{r_t}^{(t)} + \cdots + \wt k_1^{(t)}, 
          \quad k_i^{(d)} < \wt k_i^{(d)}.
\end{equation*}
Thus repeating this procedure for $t-1, \dots, 1$, we obtain 
the lemma.
\end{proof}
\para{3.11.}
We now give a proof of Proposition 2.13.
Since $|\la,\Bs\rp$ is $M$-dominant, we can write, as in 3.4, that
\begin{align*}
|\la,\Bs\rp &=\big(u_{k_1}\we \cdots \we u_{k_r}\big) \we 
u_{k_{r}+1}\we u_{k_{r}+2}\we\cdots \\
&= \b\big(u_{\Bk_{r_1}^{(1)}}^{(1)}\we \cdots 
     \we u_{\Bk_{r_l}^{(l)}}^{(l)}\bigr) 
\we u_{k_r+1} \we u_{k_r+2} \we \cdots, 
\end{align*}
where $\b \in \QQ[q,q\iv]$ and $r$ is sufficient large. 
By (2.4.1), we have 
\begin{align*}
\ol{|\la,\Bs\rp}&= \ol{\big(u_{k_1}\we \cdots \we u_{k_r}\big)} 
\we u_{k_{r}+1}\we u_{k_{r}+2}\we\cdots \\
&=\ol{\b}\, \ol{\big(u_{\Bk_{r_1}^{(1)}}^{(1)}\we \cdots \we 
                  u_{\Bk_{r_l}^{(l)}}^{(l)}\big)}
                \we u_{k_r+1} \we u_{k_r+2} \we \cdots.
\end{align*}
By (2.4.2), we have 
\begin{align*}
&\ol{u_{\Bk_{r_1}^{(1)}}^{(1)}\we \cdots \we 
              u_{\Bk_{r_l}^{(l)}}^{(l)}} \\  
& \qquad = \b'\big(u_{k_{r_l}^{(l)}}^{(l)}\we\cdots \we 
                       u_{k_1^{(t+1)}}^{(t+1)}\big)
	\we \bigl(u_{k_{r_t}^{(t)}}^{(t)}\we \cdots 
u_{k_2^{(1)}}^{(1)}\we u_{k_1^{(1)}}^{(1)}\bigr)
\end{align*}
with $\b'\in \QQ[q,q\iv]$. By using Lemma 3.10 repeatedly, we have 
\begin{align*}
\big(u_{k_{r_l}^{(l)}}^{(l)}&\we\cdots \we 
u_{k_2^{(t+1)}}^{(t+1)}\we u_{k_1^{(t+1)}}^{(t+1)}\big)
	\we \big(u_{k_{r_t}^{(t)}}^{(t)}\we \cdots \we 
u_{k_2^{(1)}}^{(1)}\we u_{k_1^{(1)}}^{(1)}\big)   \\
&=\b''\, \big(u_{k_{r_g}^{(t)}}^{(t)}\we \cdots 
u_{k_2^{(1)}}^{(1)}\we u_{k_1^{(1)}}^{(1)}\big)\we 
	\big(u_{k_{r_l}^{(l)}}^{(l)}\we\cdots \we 
u_{k_2^{(t+1)}}^{(t+1)}\we u_{k_1^{(t+1)}}^{(t+1)}\big)   + Y, 
\end{align*}
where $\b'' \in \QQ[q,q\iv]$, and $Y$ is a $\QQ[q,q\iv]$-linear combination
of the wedges of the form
\begin{equation*} 
\big(u_{\wt{k}_{r_t}^{(t)}}^{(t)}\we \cdots \we 
u_{\wt{k}_1^{(1)}}^{(1)}\big)\we 
\big(u_{\wt{k}_{r_l}^{(l)}}^{(l)}\we\cdots \we 
u_{\wt{k}_1^{(t+1)}}^{(t+1)}\big) 
\quad\text{ for }\quad
   \begin{cases}
(\wt k_{r_t}^{(t)}, \dots, \wt k_1^{(1)}) \in \ZZ^{r'}, \\
(\wt k_{r_l}^{(l)}, \dots, \wt k_1^{(t+1)}) \in \ZZ^{r''}
   \end{cases}
\end{equation*}
under the condition
\begin{equation*}
\tag{3.11.1}
\begin{aligned}
           k_1^{(1)} + \cdots + k_{r_t}^{(t)} &> 
        \wt k_1^{(1)} + \cdots + \wt k_{r_t}^{(t)}, \\
k_1^{(t+1)} + \cdots + k_{r_l}^{(l)} &< 
     \wt k_1^{(t+1)} + \cdots + \wt k_{r_l}^{(l)}.  
\end{aligned}
\end{equation*}
We claim that 
\par\medskip\noindent
(3.11.2) \ The wedges appearing in $Y$ is written as 
a linear combination of the wedges $u_{\Bh} = |\mu, \Bs\rp$ such that
$\Ba(\la) > \Ba(\mu)$ for $\mu \in \Pi^l$.
\par\medskip
We show (3.11.2).  By using the ordering rule, 
$u_{\wt{k}_{r_t}^{(t)}}^{(t)}\we \cdots \we 
u_{\wt{k}_1^{(1)}}^{(1)}$ 
(resp. $u_{\wt k_{r_l}^{(l)}}^{(l)}\we \cdots \we 
    u_{\wt k_1^{(t+1)}}^{(t+1)}$)
can be written as a linear combination of the
wedges
$u_{\Bh_{r_1}^{(1)}}^{(1)}\we\cdots\we u_{\Bh_{r_t}^{(t)}}^{(t)}$
(resp. $u_{\Bh_{r_{t+1}}^{(1)}}^{(1)}\we\cdots\we 
              u_{\Bh_{r_l}^{(l)}}^{(l)}$)
with $u_{\Bh_{r_i}^{(i)}}^{(i)} = 
   u_{h_1^{(i)}}^{(i)}\we\cdots\we u_{h_{r_i}^{(i)}}^{(i)}$.
Proposition 3.6 says that the sum of the indices is stable 
in applying the ordering rule.  Hence we have 
\begin{equation*}
\tag{3.11.3}
\begin{aligned}
\wt k_1^{(1)} + \cdots + \wt k_{r_t}^{(t)} 
   &= h_1^{(1)} + \cdots + h_{r_t}^{(t)}, \\
\wt k_1^{(t+1)} + \cdots + \wt k_{r_l}^{(l)}
   &= h_1^{(t+1)} + \cdots + h_{r_l}^{(l)}.
\end{aligned}
\end{equation*}
Recall that $k_i^{(b)} = s_b -i + 1 + \la^{(b)}_i$. 
We define $\mu \in \Pi^l$ by setting 
$h_i^{(b)} = s_b - i + 1 + \mu^{(b)}_i$ for any $i,b$, 
and write it as $\mu = (\mu^{[1]}, \mu^{[2]})$.
Then (3.11.1) and (3.11.3) imply that 
$|\la^{[1]}| > |\mu^{[1]}|$.
Also we note that $|\la| = |\mu|$ since 
\begin{equation*}
k_1^{(1)} + \cdots + k_{r_l}^{(l)}
 = \wt k_1^{(1)} + \cdots + \wt k_{r_l}^{(l)}
 = h_1^{(1)} + \cdots + h_{r_l}^{(l)}.
\end{equation*}
Hence we have $\Ba(\la) > \Ba(\mu)$.
Moreover, by using (3.4.1), we see that 
$(u_{\Bh_{r_1}^{(1)}}^{(1)}\we\cdots\we 
    u_{\Bh_{r_l}^{(l)}}^{(l)})\we u_{k_r +1} \we \cdots$
coincides with $u_{\Bh} = |\mu,\Bs\rp$ up to scalar.
Thus (3.11.2) is proved.
\par
Now as noted in Remark 3.7, the ordering rule 
for $u_{k_{r_t}^{(t)}}^{(t)}\we \cdots \we u_{k_1^{(1)}}^{(1)}$, 
regarded as an element in $\vL^{s + \hinf}$ or as an element 
in $\vL^{s' + \hinf}$,  is the same. 
Hence under the map $\F$, 
${u'}_{k_{r_t}^{(t)}}^{(t)}\we \cdots 
                \we {u'}_{k_1^{(1)}}^{(1)}$
(resp. ${u''}_{k_{r_l}^{(l)}}^{(l-t)}\we \cdots 
                \we {u''}_{k_1^{(1)}}^{(t+1)}$)
corresponds to $u_{k_{r_t}^{(t)}}^{(t)}\we \cdots 
                     \we u_{k_1^{(1)}}^{(1)}$
(resp. $u_{k_{r_l}^{(l)}}^{(l)}\we \cdots 
                     \we u_{k_1^{(t+1)}}^{(t+1)})$.
It follows that 
\begin{align*}
\big(u_{k_{r_t}^{(t)}}^{(t)}\we \cdots \we 
&u_{k_1^{(1)}}^{(1)}\big)\we 
	\big(u_{k_{r_l}^{(l)}}^{(l)}\we\cdots 
         \we u_{k_1^{(t+1)}}^{(t+1)}\big)
\we u_{r+1}\we u_{r+2}\we\cdots  \label{}\\
&=\a \ \ol{|\la^{[1]},\Bs^{[1]}\rp} \otimes
 \ol{|\la^{[2]},\Bs^{[2]}\rp} 
\end{align*}
with some $\a \in \QQ[q,q\iv]$. 
Summing up the above arguments, we have 
\begin{equation*}
\tag{3.11.4}
\ol{|\la,\Bs \rp}=\a \,
	\ol{|\la^{[1]},\Bs^{[1]}\rp }\otimes \ol{|\la^{[2]},\Bs^{[2]}\rp}
	\ +\sum_{\mu\in \Pi^l \atop \Ba(\la) > \Ba(\mu) }
          \a_{\la,\mu}\, |\mu, \Bs\rp 
\end{equation*}
with $\a, \a_{\la,\mu} \in \QQ[q,q\iv]$. 
Since the coefficient of $|\la,\Bs\rp$ in the expansion of 
$\ol{|\la,\Bs\rp}$ in terms of the ordered wedges is equal to 1, 
and similarly for $|\la^{[1]},\Bs^{[1]}\rp, |\la^{[2]},\Bs^{[2]}\rp$, 
we see that $\a = 1$ by comparing the coefficient of 
$|\la,\Bs\rp = |\la^{[1]},\Bs^{[1]}\rp\otimes |\la^{[2]},\Bs^{[2]}\rp$
in the both sides of (3.11.4).
This proves the proposition.

 
\bigskip

 \end{document}